\RequirePackage{fix-cm}
\documentclass[smallextended]{svjour3}       
\smartqed  
\usepackage{graphicx}

\usepackage{amsmath}
\usepackage{amsfonts}
\usepackage{amssymb}
\usepackage{graphicx}
\usepackage{mathrsfs}
\usepackage{tabularx}
\usepackage{multirow}
\usepackage{array}
\usepackage{longtable}
\usepackage{latexsym, graphicx, epsfig}
\usepackage{color}
\usepackage{bm}
\usepackage{cite}
\graphicspath{{figs/}}

\newcommand*{\e}{\mathop{}\!\mathrm{e}}

\newcommand{\etal}{\textit{et al.}}

\newcommand{\rd}{\mathrm{d}}

\newcommand{\R}{\mathbb{R}}

\usepackage[colorlinks,linkcolor=blue,anchorcolor=blue,citecolor=red]{hyperref}

\usepackage{bm}
\usepackage{float}
\usepackage{multirow}

\begin{document}

\title{Machine learning-based moment closure model for the linear Boltzmann equation with uncertainties
}

\titlerunning{ML-based moment closure model for the linear Boltzmann equation}        

\author{Juntao Huang \and Liu Liu \and Kunlun Qi \and Jiayu Wan}


\institute{J. Huang \at
              Department of Mathematics and Statistics, Texas Tech University, Lubbock, TX, 79409, USA. Research is partially supported by NSF DMS-2309655 and DOE DE-SC0023164. \\
              \email{juntao.huang@ttu.edu}          
           \and
           L. Liu \at
           The Chinese University of Hong Kong, Hong Kong.  L. Liu acknowledges the support by National Key R\&D Program of China (2021YFA1001200), Ministry of Science and Technology in China, Early Career Scheme (24301021) and General Research Fund (14303022 \& 14301423) funded by Research Grants Council of Hong Kong. \\
        \email{lliu@math.cuhk.edu.hk}
              \and 
              K. Qi \at 
             School of Mathematics,  University of Minnesota - Twin Cities, Minneapolis, MN, 55455 USA. \\
            \email{kqi@umn.edu, kunlunqi.math@gmail.com} 
            \and 
            J. Wan \at
          The Chinese University of Hong Kong, Hong Kong \\
         \email{jiayuwan@cuhk.edu.hk}
    }


\maketitle

\begin{abstract}
The Boltzmann equation, a fundamental equation in kinetic theory, serves as a bridge between microscopic particle dynamics and macroscopic continuum mechanics. However, deriving closed macroscopic moment systems from the Boltzmann equation remains a long-standing challenge due to the intrinsic non-closure of the moment hierarchy.
In this paper, we propose a machine learning (ML)-based moment closure model for the linear Boltzmann equation, addressing both the deterministic and stochastic settings. Our approach leverages neural networks to learn the spatial gradient of the unclosed highest-order moment, enabling effective training through natural output normalization. 
For the deterministic problem, to ensure global hyperbolicity and stability, we derive and apply the constraints that enforce symmetrizable hyperbolicity of the system. 
For the stochastic problem, we adopt the generalized polynomial chaos (gPC)-based stochastic Galerkin method to discretize the random variables, resulting in a system for which the approach in the deterministic case can be used similarly. Several numerical experiments are shown to demonstrate the effectiveness and accuracy of our ML-based moment closure model for the linear Boltzmann equation with or without uncertainties. 

\keywords{Moment Closure \and Linear Boltzmann Equation \and Machine Learning \and Uncertainty Quantification \and Hyperbolicity}
\end{abstract}

\section{Introduction}
\label{sec:intro}

\textbf{Background:} 
As the bridge between microscopic particle dynamics and macroscopic continuum mechanics, the kinetic equations have been widely used in many areas such as mechanics, rarefied gas, plasma physics, astrophysics, semiconductor device modeling, and social and biological sciences \cite{Semi-Book, Villani02}. They describe the non-equilibrium dynamics of a system composed of a large number of particles and bridge atomistic and continuum models in the hierarchy of multiscale modeling. The Boltzmann-type equation, as one of the most representative models in kinetic theory, provides a power tool to describe molecular gas dynamics, radiative transfer, plasma physics, and polymer flow \cite{A15}. They have significant impacts in designing, optimization, control, and inverse problems. For example, it can be used in the design of semiconductor devices, topology optimization of gas flow channels, or risk management in quantitative finance \cite{CS21}. Many of these applications often require finding unknown or optimal parameters in the Boltzmann-type equations or mean-field models \cite{AHP15, Caflisch, CFP13, Cheng11}.

In addition, kinetic equations typically involve various sources of uncertainty, such as modeling errors, imprecise measurements, and uncertain initial conditions. As a result, addressing uncertainty quantification (UQ) becomes essential for evaluating, validating, and improving the underlying models, underscoring our project's significance. In particular, the collision kernel or scattering cross-section in the Boltzmann equation governs the transition rates during particle collisions. Calculating this collision kernel from first principles is highly complex, and in practice, heuristic approximations or empirical data are often used, inevitably introducing uncertainties. Additionally, uncertainties may stem from inaccurate measurements of initial or boundary conditions, as well as from source terms, further compounding the uncertainties in the model. For numerical studies of the Boltzmann equation and other kinetic models with or without randomness, we refer readers to works such as \cite{HPY, JXZ15, Jin-ICM, PZ2020} and \cite{pareschi, MP06, MPR13, HQ20, HQY21}, and particularly for the linear Boltzmann equation \cite{LW2017, Poette2}. 
Among the various numerical approaches, the generalized polynomial chaos (gPC)-based stochastic Galerkin (SG) method and its variations have been widely adopted, demonstrating success in a range of applications \cite{Xiu}. Beyond numerical simulations, theoretical studies have established the stability and convergence of these methods. Spectral convergence for the gPC-based SG method was demonstrated in \cite{HJ16, LQ2024-1, LQ2024-2}, while \cite{LJ18} and \cite{DJL19} introduced a robust framework based on hypocoercivity to perform local sensitivity analysis for a class of multiscale, inhomogeneous kinetic equations with random uncertainties — approximated using the gPC-based SG method. In addition, it is also worth mentioning the recent progress in using the Stochastic Galerkin particle methods for the kinetic equations with uncertainties \cite{MNRZ2025, MPZ2023}. For further reference, we point readers to the recent collection \cite{JinPareschi} and the survey \cite{parUQ}.

An important approximation strategy in kinetic theory is given by Grad's moment method \cite{Grad, Grad-book} based on the expansion of the distribution function in Hermite polynomials, which extends the set of variables of a continuum model
beyond the fields of density, velocity, and temperature. The additional variables are given
by higher-order moments of the distribution function. Indeed, to describe the kinetic effects in highly non-equilibrium regimes, many moments are needed, which results in a large system of PDEs to be solved. It is known that Struchtrup and Torrilhon \cite{TS} have rigorously derived the regularized 13-moment equations from the Boltzmann equation of the monatomic gases. The final system consists of evolution equations for the 13 fields of density, velocity, temperature, stress tensor, and heat flux. See \cite{Levermore} for regularized moment equations and moment closure hierarchies for kinetic theory. In 1950's, Grad solved the closure problem by assuming the distribution can be expressed around the Maxwellian function. According to \cite{ST2003}, Grad's choice, or other nonlinear closures such as Pearson's \cite{Torrilhon} is problematic, as it exhibits negative values in the tail for non-vanishing heat flux, which could be the reason for the loss of hyperbolicity of the moment equations. In general, the moment closure is a challenging and important problem in the approximation theory for kinetic models. Many numerical computational methods have been developed based on the Grad moment approach \cite{CL2010, CFL_ICIAM}. However, these closure conditions may not be valid in practice, especially when there are shock profiles or complicated boundary conditions. 
We list some works here: for examples the $P_N$ model \cite{Chandrasekhar1944}, the filtered $P_N$ model \cite{RARO2013}, the positive $P_N$ model \cite{MM2010}, the entropy-based $M_N$ model \cite{AFH2019} and the $MP_N$ model \cite{FLZ2020,FLZ2020-2,LLZ2021}.
In this work, our goal is to use the machine learning approach to find an accurate closure system. 

With the recent development of data-driven methodology and machine learning (ML) techniques, some new approaches based on machine learning and neural networks have been proposed to solve the moment closure problem; since the relationship between the highest-order moment and lower-order moments is generally unknown, aside from the assumption that such a relationship exists, a neural network appears to be an ideal candidate to serve as a black-box model representing this relationship after training from the data, which are obtained by solving the kinetic equation. One of the groundbreaking frameworks is in \cite{HMME2019} by Han \etal, where they first used an autoencoder to learn a set of generalized moments that optimally represent the underlying velocity distribution, and then trained a moment closure model for these generalized moments to effectively capture the dynamics associated with the kinetic equation. By utilizing the conservation-dissipation formalism, a stable closure model is developed from irreversible thermodynamics for the Boltzmann-BGK equation in \cite{huang2021learning}, parameterized by a multilayer perception. In addition, Bois \etal in \cite{BFNV2022} introduced a nonlocal closure model for the Vlasov-Poisson system using a convolutional neural network. Furthermore, in \cite{ma2020machine, wang2020deep}, the widely recognized Hammett–Perkins Landau fluid closure model was studied by ML and neural network techniques. Recently, a neural network-based approximation to regularized entropy-based closure of the Boltzmann system was proposed in \cite{porteous2021data,schotthofer2022structure,SLFH2025}, to accurately compute the solution of the multi-dimensional moment system. 
In particular, we highlight a recent series of work \cite{Huang1,Huang2,Huang3} by Huang \etal, where they develop a new ML framework for the moment closure problem of the radiative transfer equation (RTE).

For the moment closure problems, the hyperbolicity of the derived moment system is critical to the well-posedness of the first-order partial differential equations \cite{Serre1}. In fact, Grad's pioneering work on moment closure in gas kinetic theory, presented in \cite{Grad}, laid the foundation for moment models. However, it was later shown in \cite{CFL2014KRM} that in the three-dimensional case, the equilibrium state for Grad's 13-moment model lies on the boundary of the hyperbolicity region. This limitation significantly restricts the applicability of the moment method. Consequently, this issue has garnered considerable attention, with numerous studies in the literature \cite{CFL2014CPAM, CFL2013CMS,FLZ2020} dedicated to developing globally hyperbolic moment systems. The classical philosophy in deriving and solving the moment system, instead of the kinetic equations themselves, is to pursue the balance between generic accuracy and practical computability. However, with recent advancements in ML and data-driven modeling \cite{BPK2016PNAS, HJE2018}, novel ML-based approaches \cite{HMME2019, scoggins2021machine,ma2020machine,MH2010JCP, BFNV2022} have emerged to address the moment closure problem, offering new potential for both accuracy and practicality. We refer the readers to \cite{Huang1} and the references therein for more recent progress in this field. Albeit the success of ML in the application of moment closure problems, it is worth mentioning that most of the aforementioned works do not ensure hyperbolicity or long-term stability, with the exception of the works in \cite{huang2021learning,porteous2021data,schotthofer2022structure,SLFH2025,Huang1,Huang2,Huang3}.

\textbf{Motivation and our contributions:}
In this paper, we focus on developing an ML-based moment closure model for the linear Boltzmann equation, addressing both deterministic and stochastic cases. Classical moment closure approaches approximate unclosed higher-order moments based on empirical assumptions, which may not hold in general \cite{CFL2014KRM}. Additionally, unclosed higher-order moments often vary widely in magnitude and can become very small on certain scales, such as in the optically thick regime for RTE \cite{Huang1}. This variability complicates neural network training, as the target function may be difficult to learn directly from lower-order moments without appropriate output normalization \cite{BFNV2022}. Therefore, rather than directly learning the unclosed higher-order moments \cite{HMME2019}, we opt to learn the spatial gradient of the unclosed moment using neural networks, drawing on a similar strategy proposed for RTE \cite{Huang1}.
For the deterministic case, we first derive the unclosed moment system using Hermite polynomials and their recurrence relations. A neural network incorporating gradients of lower-order moments with natural output normalization is then introduced to learn the gradient of the highest-order moment. To ensure long-term stability, we also introduce an approach inspired by \cite{Huang2, Huang3, christlieb2024bgk} that enforces global hyperbolicity in the ML-based moment closure model. This is achieved by constructing a symmetrizer (i.e., a symmetric positive definite matrix) for the closure system and deriving constraints that make the system globally symmetrizable hyperbolic.
For the stochastic case, we use the gPC-based SG method to discretize the random variable and derive a higher-dimensional deterministic moment system. The gradient-learning approach is also applied to this system, enabling an ML-based moment closure model.

\textbf{Organization of our paper:} The rest of this paper is organized as follows: We first introduce the linear Boltzmann equation and its associated moment closure problem in Section \ref{sec:Boltz}. In Section \ref{sec:ml}, we propose an ML-based moment closure model to learn the unclosed highest-order moment, where the random variables are handled by the stochastic Galerkin method. The data generation and the training of the neural networks are presented in Section \ref{sec:training}. We validate our proposed ML-based moment closure model by numerical examples in Section \ref{sec:numerical}. Some concluding remarks are given in Section \ref{sec:conclusion}.

\section{Linear Boltzmann equation and moments system}
\label{sec:Boltz}

\subsection{Linear Boltzmann equation}
\label{subsec:semi}

The linear Boltzmann equation with random parameters is given by 
\begin{equation}\label{semi-Boltzmann}
    \partial_t f(t,x,v,z) + v \partial_x f(t,x,v,z) = Q(f)(t,x,v,z),
\end{equation}
with the initial condition
\begin{equation}\label{f0}
f(0,x,v,z) = f^{0}(x,v,z),
\end{equation}
where $f=f(t,x,v,z)$ is the probability density function at time $t$ and position $x \in \mathbb{R}$, with velocity variable $v \in \mathbb{R}$ and a random vector $z \in I_{z} \subset \mathbb{R}^{d_z}$ with known probability distribution characterizing random inputs. For more details on the parameterization of random inputs by a finite-dimensional random vector, we refer the reader to the appendix. The collision operator $Q(f)$ is given by 
\begin{equation}\label{Qf}
    Q(f)(v,z) = \int_{\R} \sigma(v,v_*,z) \left[ M(v) f(v_*,z) - M(v_*)f(v,z) \right] \,\rd v_*\,,
\end{equation}
where $M(v)$ is the normalized Maxwellian distribution
\begin{equation}\label{Mv}
    M(v) := \frac{1}{(2\pi)^{1/2}} \e^{-|v|^2}.
\end{equation}

In this paper, we consider the one-dimensional in space and velocity variables, with uncertain parameters arising from: 
\begin{itemize}
\item[(i)] the initial datum $ f^{0}(x,v,z)$;
\item[(ii)] the collision kernel $\sigma(v,v_*,z) \geq 0$.
\end{itemize}
In particular, if we assume the scattering kernel to be isotropic (independent of the velocity variable), i.e., $\sigma(v,v_*,z) = \sigma(z)$, the equation becomes 
\begin{multline}\label{LB}
     \partial_t f(t,x,v,z) + v \partial_x f(t,x,v,z) \\
 = \sigma(z) M(v) \left( \int_{\R} f(t,x,v_*,z) \,\rd v_*\right) -  \sigma(z) f(t,x,v,z). 
\end{multline}

\subsection{Moments system}
\label{subsec:Moments}

We take the moments of the linear Boltzmann equation against the Hermite polynomials of $v$ in the whole space, instead of the bounded domain as for RTE \cite{Huang1}.
Denoting the $k$-th order Hermite polynomial by $H_k = H_k(v)$ for $k\geq 0$, the $k$-th order moments can be defined as 
\begin{equation}\label{mk}
    m_k(t,x,z) := \int_{\R} f(t,x,v,z) H_k(v) \,\rd v, \quad k \geq 0.
\end{equation}

\begin{remark}
    Recall the recurrence relation of the Hermite polynomials:
\begin{equation}\label{recur}
        H_{k+1}(v) = \sqrt{\frac{2}{k+1}} v H_{k}(v) - \sqrt{\frac{k}{k+1}} H_{k-1}(v)
    \end{equation}
    with $H_{0}(v)=\pi^{-\frac{1}{4}}$ and $H_{1}(v)=\sqrt{2}\pi^{-\frac{1}{4}}v$. Furthermore, the orthogonal relation with respect to the weight function $M(v)$ holds:
    \begin{equation}\label{ortho}
        \int_{\R} H_m(v) H_n(v) M(v) \,\rd v = 0, \quad \text{for all} \quad m \neq n.
    \end{equation}
\end{remark}

Hence, multiplying both sides of \eqref{LB} by $H_k(v)$ and integrating over the whole velocity space $\mathbb{R}$ lead to 
\begin{equation}\label{Hf}
\begin{aligned}
    &\partial_t \int_{\R} H_k(v) f(t,x,v,z) \,\rd v + \partial_x \int_{\R} v H_k(v) f(t,x,v,z) \,\rd v \\[3pt]
     =& \sigma(z) \pi^{\frac{1}{4}} \int_{\R} H_0(v) H_k(v) M(v) \,\rd v \left( \int_{\R} f(t,x,v_*,z) \,\rd v_*\right)\\
     & \qquad - \sigma(z) \int_{\R} H_k(v) f(t,x,v,z) \,\rd v,
\end{aligned}
\end{equation}
which, by considering the definition of \eqref{mk}, involving the recurrence relation \eqref{recur} on the left-hand side and the orthogonal relation \eqref{ortho} on the right-hand side of \eqref{Hf} above, can be further simplified as 
\begin{multline}
   \partial_t m_k(t,x,z) + \partial_x \left[ \sqrt{\frac{k+1}{2}} m_{k+1}(t,x,z) + \sqrt{\frac{k}{2}} m_{k-1}(t,x,z) \right]\\
     = \delta_{0k}\sigma(z)m_0(t,x,z) - \sigma(z)m_k(t,x,z).
\end{multline}
Therefore, the moment system up to $m_N$ is presented as follows:
\begin{equation}\label{moments1}
\left\{
    \begin{aligned}
        \partial_t m_0  + \sqrt{\frac{1}{2}} \partial_x  m_{1}   & = 0, \\
        \partial_t m_1  +  \partial_x  m_{2}  + \sqrt{\frac{1}{2}} \partial_x m_{0}  & = - \sigma(z)m_1, \\
        \partial_t m_2  +  \sqrt{\frac{3}{2}} \partial_x m_{3}  + \partial_x m_{1}  & = - \sigma(z)m_2, \\
        \cdots &\\
         \cdots &\\
          \cdots &\\
        \partial_t m_N  +  \sqrt{\frac{N+1}{2}} \partial_x m_{N+1}  + \sqrt{\frac{N}{2}} \partial_x m_{N-1}  &= - \sigma(z)m_N.
    \end{aligned}
    \right.
\end{equation}

We can find that, in the last equation of \eqref{moments1} above, the evolution of $N$-th order moment $m_N$ depends on $m_{N+1}$, therefore, the moments system \eqref{moments1} is unclosed. In fact, there are many classical ways to close the system, where the $P_N$ model \cite{Chandrasekhar1944} is the most straightforward approach. The $P_N$ model utilizes the orthogonal polynomials in the velocity space and assumes $m_{N+1} = 0$ to close the model, such that the system \eqref{moments1} can be written in the following vector form:
by denoting $\bm{m} = (m_0, m_1,..., m_N)^T$,
\begin{equation}\label{M-PN}
    \partial_t \bm{m} + A \partial_x \bm{m} = S\bm{m}, 
\end{equation}
where the diagonal coefficient matrix $S \in \mathbb{R}^{(N+1)\times (N+1)} $ is
\begin{equation}
    S := \text{diag} (0, -\sigma(z), -\sigma(z), \cdots, -\sigma(z))\,,
\end{equation}
and the coefficient matrix $A \in \mathbb{R}^{(N+1)\times (N+1)}$ is
\begin{equation}\label{eq:matrix-A-PN}
   A :=  
   \begin{pmatrix}
        0 & \sqrt{\frac{1}{2}} & 0 & 0 & 0 & 0 & \cdots & 0\\
        
        \sqrt{\frac{1}{2}} & 0 & 1 & 0 & 0 & 0 & \cdots & 0\\
        
        0 & 1 & 0 & \sqrt{\frac{3}{2}} & 0 & 0 & \cdots & 0\\
        
        \vdots & \vdots & \vdots & \vdots &\ddots & \ddots &\ddots & \vdots\\

        \vdots & \vdots & \vdots & \vdots &\ddots & \ddots &\ddots & \vdots\\
        
        0 & 0 & \cdots & 0 & \sqrt{\frac{N-2}{2}} & 0 & \sqrt{\frac{N-1}{2}} & 0 \\
        
        0 & 0 & \cdots & 0 & 0 & \sqrt{\frac{N-1}{2}} & 0 & \sqrt{\frac{N}{2}} \\
        
        0 & 0 & \cdots & 0 & 0 & 0 & \sqrt{\frac{N}{2}} & 0\\
    \end{pmatrix}.
\end{equation}

\section{Machine learning based moments closure model}
\label{sec:ml}

Recent advancements in ML techniques have led to notable progress in using ML frameworks to enhance moment closure models.
One of the standard ways for the moments closure is to seek the relation between the highest moment $m_{N+1}$ and the lower-order moments \cite{HMME2019}:

\begin{equation}\label{LM}
    m_{N+1} = \mathcal{N}(m_0, m_1, \cdots, m_N),
\end{equation}

where $\mathcal{N}: \mathbb{R}^{N+1} \mapsto \mathbb{R}$ is a neural network trained from data.
This is the so-called Learning the Moment (LM) approach. In \cite{HMME2019}, it served as the regression in supervised learning and a part of the end-to-end learning procedure. 


\subsection{Formulation and hyperbolic condition}
\label{subsec:hyperbolic}

In order to close the moment system \eqref{moments1} and circumvent the challenge in the LM framework, where the training process often converges to a local minimum, we will adopt the closure relation introduced in \cite{Huang1}. This approach assumes a linear relationship between the gradient of the highest moment, $\partial_{x} m_{N+1}$, and the gradients of lower-order moments, $\partial_{x} m_0, \cdots, \partial_{x} m_N$, as follows:
\begin{equation}\label{LG-deter}
     \partial_{x}m_{N+1} = \sum_{i=0}^{N} \mathcal{N}_i(m_0,m_1,...,m_N)\partial_{x}m_i.
\end{equation}
In this case, we can rewrite the moments system \eqref{moments1} in the following vector form:
\begin{equation}\label{moment-equation-LG}
    \partial_t \bm{m} + A \partial_x \bm{m} = S\bm{m}, 
\end{equation}
where the diagonal coefficient matrix $S \in \mathbb{R}^{(N+1)\times (N+1)} $ is
\begin{equation}
    S := \text{diag} (0, -\sigma(z), -\sigma(z), \cdots, -\sigma(z))\,,
\end{equation}
and the coefficient matrix $A \in \mathbb{R}^{(N+1)\times (N+1)}$ is
\begin{equation}\label{eq:matrix-A-LG}
   A :=  
   \begin{pmatrix}
        0 & \sqrt{\frac{1}{2}} & 0 & 0 & 0 & 0 & \cdots & 0\\
        
        \sqrt{\frac{1}{2}} & 0 & 1 & 0 & 0 & 0 & \cdots & 0\\
        
        0 & 1 & 0 & \sqrt{\frac{3}{2}} & 0 & 0 & \cdots & 0\\
        
        \vdots & \vdots & \vdots & \vdots &\ddots & \ddots &\ddots & \vdots\\

        \vdots & \vdots & \vdots & \vdots &\ddots & \ddots &\ddots & \vdots\\
        
        0 & 0 & \cdots & 0 & \sqrt{\frac{N-2}{2}} & 0 & \sqrt{\frac{N-1}{2}} & 0 \\
        
        0 & 0 & \cdots & 0 & 0 & \sqrt{\frac{N-1}{2}} & 0 & \sqrt{\frac{N}{2}} \\
        
        a_0 & a_1 & \cdots & a_{N-4} & a_{N-3} & a_{N-2} & a_{N-1} & a_{N}\\
    \end{pmatrix}
\end{equation}
with coefficients $a_i$ in the last row of $A$ and $\mathcal{N}_i$ satisfying the following relation:
\begin{equation}\label{coef-relation}
a_j = 
\begin{cases}
    \sqrt{\frac{N+1}{2}}\mathcal{N}_j, & j \neq N-1, \\[4pt]
    \sqrt{\frac{N+1}{2}}\mathcal{N}_j + \sqrt{\frac{N}{2}}, & j = N-1.
\end{cases}
\end{equation}

In general, the matrix $A$ is not real-diagonalizable, so the system is not necessarily hyperbolic. We are trying to find a condition that enforces $A$ to be hyperbolic, so that the system \eqref{moment-equation-LG} remains stable over time. To achieve this, we follow the technique introduced in \cite{Huang2}, i.e., we seek an SPD matrix $A_0$ such that $A_{0}A$ is symmetric. However, this matrix $A_0$ is usually hard to compute. Therefore, without loss of generality, we relax the assumption \eqref{LG-deter} by removing the first $k$ dependence that
\[ 
\partial_{x}m_{N+1} = \sum_{i=N-k}^{N} \mathcal{N}_i(m_0,m_1,...,m_N)\partial_{x}m_i,  
\]
where $k=2,3$ are typical choices to simplify the computation, and in either case, we assume $N \geq 3$ to avoid the trivial results. 
In what follows, we will take $k=2$ to illustrate the idea, whereas the same strategy can be extended directly to $k=3$. 

\begin{theorem}\label{theorem-symmetric}
    Consider the matrix $A \in \mathbb{R}^{(N+1)\times (N+1)}$ with $N \geq 3$ and $k=2$ in \eqref{eq:matrix-A-LG}, i.e.,
    \begin{equation*}
    A =  
   \begin{pmatrix}
        0 & \sqrt{\frac{1}{2}} & 0 & 0 & 0 & 0 & \cdots & 0\\
        
        \sqrt{\frac{1}{2}} & 0 & 1 & 0 & 0 & 0 & \cdots & 0\\
        
        0 & 1 & 0 & \sqrt{\frac{3}{2}} & 0 & 0 & \cdots & 0\\
        
        \vdots & \vdots & \vdots & \vdots &\ddots & \ddots &\ddots & \vdots\\

        \vdots & \vdots & \vdots & \vdots &\ddots & \ddots &\ddots & \vdots\\
        
        0 & 0 & \cdots & 0 & \sqrt{\frac{N-2}{2}} & 0 & \sqrt{\frac{N-1}{2}} & 0 \\
        
        0 & 0 & \cdots & 0 & 0 & \sqrt{\frac{N-1}{2}} & 0 & \sqrt{\frac{N}{2}} \\
        
        0 & 0 & \cdots & 0 & 0 & a_{N-2} & a_{N-1} & a_{N}\\
    \end{pmatrix}
    \end{equation*}
    where $a_j$ is defined as in \eqref{coef-relation}. If the constraints of the coefficients $a_j, j = N-2, N-1, N$  hold:
    \begin{equation}\label{a-inequality}
    \sqrt{\frac{N-1}{2}}a_{N-1}+a_Na_{N-2} - \sqrt{\frac{N}{2}}a_{N-2}^2 >0,
    \end{equation}
    or equivalently, for $\mathcal{N}_j, j = N-2, N-1, N$, 
    \begin{equation}\label{hyper-inequality}
    \frac{N+1}{2}\left(\mathcal{N}_{N-2}\mathcal{N}_{N} - \sqrt{\frac{N}{2}}\mathcal{N}_{N-2}^2 \right)+\sqrt{\frac{N-1}{2}} \left(\sqrt{\frac{N+1}{2}}\mathcal{N}_{N-1}+\sqrt{\frac{N}{2}}\right) > 0,
    \end{equation}
    then there exists an SPD matrix $A_0 = 
    \begin{pmatrix} 
    I & 0 \\
    0 & B
    \end{pmatrix}$ 
    such that $A_0A$ is symmetric, where $I \in \mathbb{R}^{(N-1)\times(N-1)}$ is an identity matrix, and $B = 
    \begin{pmatrix} 
    b_1 & b_2 \\
    b_2 & b_3
    \end{pmatrix}$ is an SPD matrix.
\end{theorem}

\begin{proof}
Since $A_0 A$ is required to be a real symmetric matrix, we perform the matrix multiplication and impose the condition that the corresponding entries be equal, thereby ensuring that $A_0 A$ is symmetric. This leads to the following equations:
\begin{equation}\label{bmatrix}
\left\{
    \begin{aligned}
        \sqrt{\frac{N-1}{2}}b_1 +a_{N-2}b_2= & \, \sqrt{\frac{N-1}{2}}, \\[3pt]
        \sqrt{\frac{N-1}{2}}b_2 + a_{N-2}b_3= &\, 0, \\[3pt]
        \sqrt{\frac{N}{2}}b_1 + a_Nb_2=& \, a_{N-1}b_3,
    \end{aligned}
\right.
\end{equation}
which can be further written in matrix form $\bm{M}\bm{b}=\bm{c}$ with
\begin{equation*}
\bm{M}=\begin{pmatrix} 
\sqrt{\frac{N-1}{2}} & a_{N-2} & 0  \\
0 & \sqrt{\frac{N-1}{2}} & a_{N-2} \\
\sqrt{\frac{N}{2}} & a_N & -a_{N-1}
\end{pmatrix}, \quad 
\bm{b}=
\begin{pmatrix} 
b_1 \\[8pt]
b_2 \\[8pt]
b_3 
\end{pmatrix}, \quad
\bm{c}=
\begin{pmatrix}
\sqrt{\frac{N-1}{2}} \\[8pt]
0 \\[8pt]
0 
\end{pmatrix}. 
\end{equation*}
Then, using the Cramer's rule to solve \eqref{bmatrix} above, we find
\begin{equation}\label{b-values}
\left\{
    \begin{aligned}
        b_1= &\frac{\sqrt{\frac{N-1}{2}}(-\sqrt{\frac{N-1}{2}}a_{N-1}-a_Na_{N-2})}{\sqrt{\frac{N-1}{2}}(-\sqrt{\frac{N-1}{2}}a_{N-1}-a_Na_{N-2})+\sqrt{\frac{N}{2}}a_{N-2}^2}, \\
        b_2= &\frac{\sqrt{\frac{N-1}{2}}\sqrt{\frac{N}{2}}a_{N-2}}{\sqrt{\frac{N-1}{2}}(-\sqrt{\frac{N-1}{2}}a_{N-1}-a_Na_{N-2})+\sqrt{\frac{N}{2}}a_{N-2}^2}, \\
        b_3=&\frac{-\sqrt{\frac{N}{2}}\sqrt{\frac{N-1}{2}}}{\sqrt{\frac{N-1}{2}}(-\sqrt{\frac{N-1}{2}}a_{N-1}-a_Na_{N-2})+\sqrt{\frac{N}{2}}a_{N-2}^2}.
    \end{aligned}
\right.
\end{equation}

Considering the fact that $B$ is an SPD matrix, by Sylvester's criterion, it has to satisfy the following inequalities: 
\begin{equation}\label{b-inequalities}
\left\{
    \begin{aligned}
        &b_1 > 0, \\
        &b_1b_3 - b_{2}^2 > 0,
    \end{aligned}
    \right.
\end{equation}
from which, it is clear that $b_3 >0$, i.e.,
\begin{equation*}
    \sqrt{\frac{N-1}{2}} \left(-\sqrt{\frac{N-1}{2}}a_{N-1}-a_Na_{N-2}\right) + \sqrt{\frac{N}{2}}a_{N-2}^2 < 0.
\end{equation*}

Since $b_1 >0$, we must have $\sqrt{\frac{N-1}{2}}a_{N-1}+a_Na_{N-2}>0$. On the other hand, to achieve $b_1b_3 - b_{2}^2>0$, we have $\sqrt{\frac{N-1}{2}}a_{N-1}+a_Na_{N-2} - \sqrt{\frac{N}{2}}a_{N-2}^2 >0$. Hence, considering $N \geq 3$, all these inequalities can be summarized into a single constraint \eqref{a-inequality}.

Finally, by substituting \eqref{coef-relation} into \eqref{a-inequality}, we can obtain the constraint \eqref{hyper-inequality} that should be satisfied by $\mathcal{N}_j$, $j = N-2, N-1, N$.
\end{proof}

\begin{remark}
    If $N=1$, we have no other option but to set $k=1$. In this case, the hyperbolic constraint is given by $a_0 >0$, or equivalently, $\mathcal{N}_{0} > -\sqrt{\frac{N}{N+1}}$. 
\end{remark}


\subsection{Stochastic Galerkin (SG) method for random variable}
\label{subsec:SG}
To numerically discretize the random vector $z$ in our moment system \eqref{moments1}, we apply the stochastic Galerkin (SG) method to eliminate the randomness, thereby transforming the system into an equivalent form containing only deterministic coefficients. To simplify our arguments, we assume $z \in I_z \subset \mathbb{R}^{d_z}$ with $d_z=1$ throughout the rest of this paper, which can be generalized to high dimension without an essential difference.

We define the space 
$$\mathbb{P}^K := \text{Span} \left\{ \phi_{i}(z) \, \Big| \, 0 \leq i \leq K \right\}$$
equipped with the inner product with respect to the probability density function $\pi(z)$ in $z$:
\begin{equation*}
	\langle f(t,x,v,\cdot), \, g(t,x,v,\cdot) \rangle_{I_{z}} = \int_{I_{z}}  f(t,x,v,z) \, \overline{g(t,x,v,z)} \, \pi(z) \,\rd z,
\end{equation*}
where $\{ \phi_{i}(z)\}_{i = 0}^{K}$ is an orthonormal gPC basis function, i.e.,
\begin{equation}\label{orthonormal-relation}
    \int_{I_z} \phi_{i}(z) \phi_{j}(z) \pi(z) \,\rd z = \delta_{ij}, \quad 0 \leq i, j \leq K.
\end{equation}
Then, the typical SG method is based on seeking an approximation of $f(z)$ in $\mathbb{P}^{K}$ such that
\begin{equation}\label{Galerkin-expansion}
    f(z) \approx \sum_{i=0}^{K} f^{i} \phi_{i}(z) \quad \text{with} \quad f^{i} = \int_{I_z} f(z) \phi_{i}(z) \pi(z) \,\rd z.
\end{equation}

Now, in the case of \eqref{moments1}, we can expand $m_{k}(t,x,z)$ as follows: 
\begin{equation}\label{moment-expansion}
    m_{k}(t,x,z) \approx \sum_{i=0}^{K} m_{k}^{i}(t,x) \phi_{i}(z)
\end{equation}
with
\begin{equation}\label{mki}
    m_{k}^{i}(t,x)= \int_{I_z} m_{k}(t,x,z) \phi_{i}(z) \pi(z) \,\rd z.
\end{equation}
More precisely, considering the $k$-th equation in \eqref{moments1},
\begin{multline}\label{kth-equation}
    \partial_t m_k(t,x,z) +  \sqrt{\frac{k+1}{2}} \partial_x m_{k+1}(t,x,z) + \sqrt{\frac{k}{2}} \partial_x m_{k-1}(t,x,z)\\  = - \sigma(z)m_k(t,x,z),
\end{multline}
and following \eqref{moment-expansion}, we can project both sides of \eqref{kth-equation} into $\mathbb{P}^{K}$ and obtain
\begin{multline}\label{kth-equation-Galerkin}
    \partial_t m_{k}^{i}(t,x) +  \sqrt{\frac{k+1}{2}} \partial_x m_{k+1}^{i}(t,x) + \sqrt{\frac{k}{2}} \partial_x m_{k-1}^{i}(t,x) \\= \sum_{j=0}^{K} S_{ij} m_{k}^{j}(t,x),
\end{multline}
for $k\geq 1$ and $0 \leq i \leq K$, where the matrix $S=(S_{ij}) \in \mathbb{R}^{(K+1) \times (K+1)}$ includes pre-computed weights concerning the random collision kernel $\sigma(z)$ as follows:
\begin{equation*}\label{Sij}
    S_{ij} =\int_{I_z} - \sigma(z) \phi_{j}(z) \phi_{i}(z) \pi(z) \,\rd z.
\end{equation*}

Furthermore, by denoting $\bm{m}_{k}(t,x)=(m_{k}^{0}(t,x), m_{k}^{1}(t,x),...,m_{k}^{K}(t,x))^{T}$, we can rewrite \eqref{kth-equation-Galerkin} in the following vector form,
\begin{equation}\label{kth-equation-vector-form}
    \partial_t \bm{m}_{k} + \sqrt{\frac{k+1}{2}} \partial_x \bm{m}_{k+1} + \sqrt{\frac{k}{2}} \partial_x \bm{m}_{k-1}= \bm{S}\bm{m}_{k},
\end{equation}
for $k \geq 1$, and 
\begin{equation*}
    \partial_t \bm{m}_{0} + \sqrt{\frac{1}{2}} \partial_x \bm{m}_{1} = \bm{0},
\end{equation*}
for $k=0$.

Again, we need to propose a closure relation before we can solve the system. We follow the same dependence as in \eqref{LG-deter} and assume: 
\begin{equation}\label{LG-random}
     \partial_{x}\bm{m}_{N+1} = \sum_{i=0}^{N} \bm{\mathcal{N}}_i(\bm{m}_0,\bm{m}_1,...,\bm{m}_N) \partial_{x}\bm{m}_i.
\end{equation}
Then, by inserting \eqref{LG-random} into \eqref{kth-equation-vector-form} and denoting $\bm{\mathfrak{m}}=(\bm{m}_{0},\bm{m}_{1},...,\bm{m}_{N})^{T}$, the moment system via SG method can be written as: 
\begin{equation}\label{moments-equation-random}
    \partial_t \bm{\mathfrak{m}} + \bm{A} \partial_x \bm{\mathfrak{m}} = \bm{S}\bm{\mathfrak{m}},
\end{equation}
where
\begin{equation*}
\bm{A} = 
  \begingroup 
\setlength\arraycolsep{-2.5pt}
 \begin{pmatrix}
        0 & \sqrt{\frac{1}{2}} \bm{I}_{K+1} & 0 & 0 & 0 & 0 & \cdots & 0\\
        
        \sqrt{\frac{1}{2}}\bm{I}_{K+1} & 0 & \bm{I}_{K+1} & 0 & 0 & 0 & \cdots & 0\\
        
        0 & \bm{I}_{K+1} & 0 & \sqrt{\frac{3}{2}}\bm{I}_{K+1} & 0 & 0 & \cdots & 0\\
        
        \vdots & \vdots & \vdots & \vdots &\ddots & \ddots &\ddots & \vdots\\

        \vdots & \vdots & \vdots & \vdots &\ddots & \ddots &\ddots & \vdots\\
        
        0 & 0 & \cdots & 0 & \sqrt{\frac{N-2}{2}}\bm{I}_{K+1} & 0 & \sqrt{\frac{N-1}{2}}\bm{I}_{K+1} & 0 \\
        
        0 & 0 & \cdots & 0 & 0 & \sqrt{\frac{N-1}{2}}\bm{I}_{K+1} & 0 & \sqrt{\frac{N}{2}}\bm{I}_{K+1} \\
        
        \bm{a_0} & \bm{a_1} & \cdots & \bm{a_{N-4}} & \bm{a_{N-3}} & \bm{a_{N-2}} & \bm{a_{N-1}} & \bm{a_{N}}
    \end{pmatrix}
    \endgroup
\end{equation*}
with
\begin{equation*}\label{coef-relation_sto}
\bm{a}_j = 
\begin{cases}
    \sqrt{\frac{N+1}{2}}\bm{\mathcal{N}}_j, & j \neq N-1, \\[4pt]
    \sqrt{\frac{N+1}{2}}\bm{\mathcal{N}}_j + \sqrt{\frac{N}{2}}, & j = N-1.
\end{cases}
\end{equation*}
and
\begin{equation*}
    \bm{S} = \text{diag} ( 0, \, S, \, S,  \, \cdots, \, S).
\end{equation*}

In fact, when applying the SG method, we are often interested in the moments' expectation $\mathbb{E}(m_k)$ and standard deviation $s(m_k)$, which are closely related to the coefficients in \eqref{moment-expansion}. Without loss of generality, we assume $\phi_{0}(z)=1$ such that for each moment $m_k$, we have,
\begin{equation}\label{moment-expectation}
    \begin{aligned}
         \mathbb{E}(m_k) & \approx ~\mathbb{E} \left(\sum_{i=0}^{K} m_{k}^{i}\phi_{i}\right) \\
         =& \int_{I_z} \sum_{i=0}^{K} m_{k}^{i}(t,x)\phi_{i}(z) \pi(z) \, \rd z\\[4pt]
        & = \sum_{i=0}^{K}m_{k}^{i}(t,x) \int_{I_z} \phi_{0}(z)\phi_{i}(z) \pi(z) \, \rd z
        = m_{k}^{0}(t,x),
    \end{aligned}
\end{equation}
where we apply the orthonormality of $\{ \phi_{i}(z) \}$ in the last equality above. 

For the standard deviation, we have,
\begin{equation}\label{moment-std}
\begin{aligned}
        s(m_k) 
        \approx  \sqrt{\mathbb{E}((\sum_{i=0}^{K} m_{k}^{i}\phi_{i})^2) -(m_{k}^{0})^{2}}
         =\sqrt{\sum_{i=1}^{K} (m_{k}^{i})^2}. 
 \end{aligned}
\end{equation}

\section{Training and methodology}
\label{sec:training}

In this section, we present the details about learning $\partial_{x}m_{N+1}$ by the lower orders of moments, and using the WENO scheme to solve the system after $\partial_{x}m_{N+1}$ is properly approximated. We will deal with both the deterministic case and the corresponding UQ problem. Numerical results of both cases will be presented in the next Section \ref{sec:numerical}.

\subsection{Data preparation}
\label{subsec:data prep}

One key ingredient of our methodology is to approximate the highest moment using the lower orders of moments. To achieve this goal, we need to train a neural network $\mathcal{N}=(\mathcal{N}_0,\cdots,\mathcal{N}_N)$: $\mathbb{R}^{N+1} \rightarrow \mathbb{R}^{N+1}$ as in \eqref{LG-deter} for the deterministic problem or a network $\bm{\mathcal{N}}=(\bm{\mathcal{N}}_0,\cdots,\bm{\mathcal{N}}_N): \mathbb{R}^{(K+1) \times (N+1)} \rightarrow \mathbb{R}^{(K+1) \times (N+1)}$ as in \eqref{LG-random} for the stochastic case. The first step of our training process is to prepare training data to fit in our models. We will use synthetic data, ie: reference solutions for the moments obtained from classical numerical algorithms, to serve as the input and labels of our networks. 

In the case of the deterministic model \eqref{LG-deter}-\eqref{moment-equation-LG}, we apply the parity-based method given by Jin and Pareschi in \cite{JP2000} to solve the deterministic counterpart of \eqref{semi-Boltzmann} (no $z$) and obtain the reference solution $f(t,x,v)$. For simplicity, we consider the one dimension in space $x$ and velocity $v$ for illustration, where we set the computational domain of $x$ to be $[0,1]$ with grid points $N_{x}=100$, and apply $N_v = 8$ for velocity discretization. Following the CFL condition, the time step size is chosen as $\Delta t =0.1\Delta x$ with the final time $t=0.5$. 
Once $f(t,x,v)$ is obtained, we compute the $k$-th moment $m_{k}(t,x)$ by integrating $f$ against the corresponding Hermite polynomial $H_{k}$, as introduced in \eqref{mk}, where the Gauss-Hermite quadrature rule with $N_v = 8$ is used for integral evaluation. 

In the case of the moment system with uncertainty \eqref{LG-random}-\eqref{moments-equation-random}, we need to compute reference solutions $m_{k}^{i}(t,x)$, which is the $i$-th Galerkin coefficient of the $k$-th moment as defined in \eqref{moment-expansion}. Based on the stochastic collocation (SC) method \cite{Xiu}, we can obtain the coefficients $m_{k}^{i}(t,x)$ as in \eqref{mki}. To this end, the integral in \eqref{mki} is evaluated by
\begin{equation}\label{SC-method}
    \int_{I_z} m_{k}(t,x,z) \phi_{i}(z) \pi(z) \, \rd z \approx \sum_{j=1}^{M} m_{k}(t,x,z_j) \phi_{i}(z_j) \pi(z_j) w_j, 
\end{equation}
where $(z_j,w_j)$ are the collocation points and corresponding weights with $M$-quadrature nodes, and $m_{k}(t,x,z_j)$ are obtained similarly as in the deterministic case at each $z_j$.

\subsection{Training}
\label{subsec:traning}

\begin{figure}[!htbp]
    \center
    \includegraphics[width=1\textwidth]{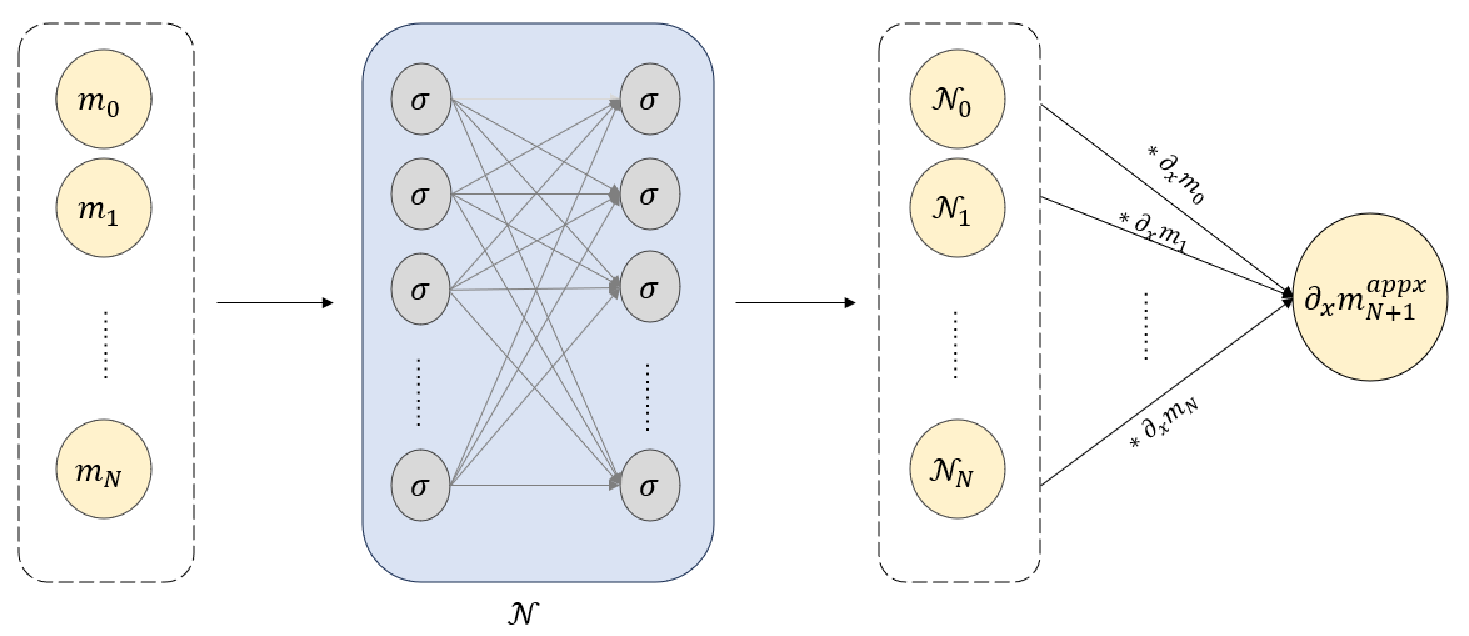}
    \caption{Architecture of our neural network.}
    \label{standard-DNN}
\end{figure}
In this subsection, we discuss the details of the architecture and the training process of the neural networks ($\mathcal{N}$ or $\bm{\mathcal{N}}$) mentioned in the subsection above. The architecture (Fig.~\ref{standard-DNN}) we choose is a standard fully connected neural network, where the input consists of lower moments (or their Galerkin coefficients in the UQ setting). This network is designed with 5 hidden layers, each containing 256 nodes, and employs the ReLU activation function. The output dimension matches that of the input. Fig.~\ref{standard-DNN} provides a graphical representation of this architecture. If hyperbolic condition is considered, we follow the same construction in \cite{Huang2} to modify the output layer to incorporate hyperbolicity into our model.

To train the neural networks, we apply the Adam optimizer with the learning rate $10^{-3}$ initially. The total number of training epochs is  1000 and the learning rate is set to decrease to $0.35$ every 100 epochs. We let the batch size be 1024. 
The input is normalized with zero mean and unit variance. These training hyperparameters are used in both deterministic and stochastic tests, and the only difference between the networks $\mathcal{N}$ and $\bm{\mathcal{N}}$ is the size of the input and output. We use 90\% of the data to train the networks and the rest of the data for validation. The hyperparameters and the activation function are tuned to minimize the loss function, which we describe below. 

In the last subsection above, we have discussed how to obtain the reference solution for $m_k, \partial_{x} m_k$ in the deterministic problem and $m_{k}^{i}, \partial_{x} m_{k}^{i}$ in the stochastic problem, for which we denote ``true" in the superscript as follows: 
\begin{equation}\label{moment-notations}
    \begin{split}
        \bm{m}_{det}^{true}=& (m_{0}^{true},\cdots,m_{N}^{true}),\\[3pt]
        \bm{\mathfrak{m}}_{sto}^{true}=& (\bm{m}_{0}^{true},\cdots,\bm{m}_{N}^{true}) \ \text{with} \ \bm{m}_{k}^{true} = (m_{k}^{0,true}, \cdots, m_{k}^{K,true}), \\[5pt]
        \partial_{x} \bm{m}_{det}^{true}=& (\partial_{x} m_{0}^{true},\cdots,\partial_{x} m_{N}^{true}),\\[3pt]
        \partial_{x} \bm{\mathfrak{m}}_{sto}^{true}=&(\partial_{x} \bm{m}_{0}^{true},\cdots,\partial_{x} \bm{m}_{N}^{true}) \ \text{with} \ \partial_{x} \bm{m}_{k}^{true} =(\partial_{x} m_{k}^{0,true},\cdots,\partial_{x} m_{k}^{K,true}).
    \end{split}
\end{equation}

Then, in the architecture shown in Fig.~\ref{standard-DNN}, the $\partial_{x} m_{N+1}^{appx}$ stand for the approximation from the neural network $\mathcal{N}$ in the deterministic case: following \eqref{LG-deter}, 
\begin{equation}\label{appx-standard}
    \partial_{x} m_{N+1}^{appx}(x_j,t_n) = \langle\mathcal{N}(\bm{m}_{det}^{true}(x_j,t_n)), \partial_{x} \bm{m}_{det}^{true}(x_j,t_n)\rangle, 
\end{equation}
where $\langle \cdot \rangle$ is the inner product in the standard Euclidean space. The approximation in the stochastic case is denoted by $\partial_{x} \bm{m}_{N+1}^{appx}(x_j,t_n)$ in a similar manner.

Now, we are in a position to introduce the loss functions for our neural networks in both the deterministic and the stochastic settings: 
\begin{equation}\label{loss-det}
    \mathcal{L}_{det} := \frac{1}{N_{data}} \sum_{j,n} |\partial_{x} m_{N+1}^{true}(x_j,t_n) - \partial_{x} m_{N+1}^{appx}(x_j,t_n)|^{2}, 
\end{equation}
\begin{equation}\label{loss-sto}
    \mathcal{L}_{sto} := \frac{1}{N_{data}} \sum_{j,n} \|\partial_{x} \bm{m}_{N+1}^{true}(x_j,t_n) - \partial_{x} \bm{m}_{N+1}^{appx}(x_j,t_n)\|_2^{2}, 
\end{equation}
and we will measure the accuracy of our moment closure models by evaluating the relative $L^2$ errors between the approximated solutions by neural networks and reference solutions by solving the kinetic equation as follows:
\begin{equation}\label{relative-L2-det}
    E_{2,det} := \sqrt{\frac{\sum_{j,n} |\partial_{x} m_{N+1}^{true}(x_j,t_n) - \partial_{x} m_{N+1}^{appx}(x_j,t_n)|^{2}}{\sum_{j,n} |\partial_{x} m_{N+1}^{true}(x_j,t_n)|^{2}}}, 
\end{equation}
\begin{equation}\label{relative-L2-sto}
    E_{2,sto} = \sqrt{\frac{\sum_{j,n} \|\partial_{x} \bm{m}_{N+1}^{true}(x_j,t_n) - \partial_{x} \bm{m}_{N+1}^{appx}(x_j,t_n)\|_2^{2}}{\sum_{j,n} \|\partial_{x} \bm{m}_{N+1}^{true}(x_j,t_n)\|_2^{2}}}. 
\end{equation}

In Fig.~\ref{relative_l2_various_sigma_N3_N5}, we present the relative $L^2$ errors with respect to epochs by using our architecture to train $\partial_{x} m_{N+1}$ in the deterministic problem. We compare the performances when $N=3,5$ with $\sigma =2$ and $\sigma=10$. We observe that for both values of $N$ and $\sigma$, the relative training errors saturated at about $0.01 \sim 0.04$, with a comparably fast converging rate(around 300 epochs). 

\begin{figure}[!htbp]
\center
    \includegraphics[width=1\textwidth]{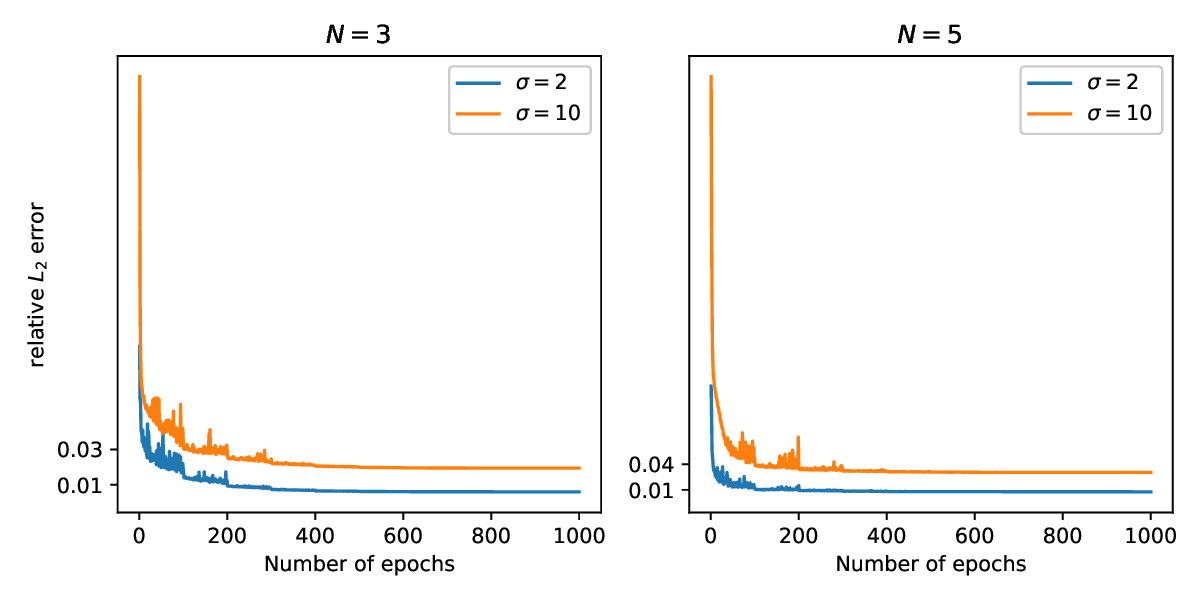}
    \caption{Left: relative $L^2$ errors of \eqref{relative-L2-det} for $N=3$ with $\sigma=2,10$. Right: relative $L^2$ errors of \eqref{relative-L2-det} for $N=5$ with $\sigma=2,10$ }
   \label{relative_l2_various_sigma_N3_N5}
\end{figure}

\subsection{WENO Scheme}

Once $\partial_{x} m_{N+1}$ is properly learned using the lower-order moments, we need to solve the systems \eqref{moment-equation-LG} and \eqref{moments-equation-random} using some classical numerical schemes. The scheme we choose is the fifth-order finite difference WENO scheme with a Lax-Friedrichs flux-splitting for spatial discretization \cite{jiang1996efficient}. We take the grid number in space to be $N_x=100$. For the time discretization, we apply the third-order strong-stability-preserving Runge-Kutta (RK) scheme \cite{shu1988efficient} with CFL condition  $\Delta t= 0.1\Delta x$. The penalty constant in the Lax-Friedrichs numerical flux is chosen to be $\alpha_{LF}=5$.

\section{Numerical results}
\label{sec:numerical}

\subsection{\textbf{Test I: Deterministic Problem}}

In this section, we present some results of our numerical experiments. To simplify the notation, we use the terms $P_N$, LM, LG, LG-hyper for moment closure methods corresponding to $m_{N+1}=0$, \eqref{LM}, \eqref{LG-deter} and \eqref{LG-deter} with hyperbolic constraints(discussed in Subsection 3.1), respectively. We use the term "real" to stand for reference solutions obtained by parity-based method(deterministic case) and stochastic collocation method(stochastic case), as described in subsection 4.1. 

We first examine the deterministic case, where \eqref{semi-Boltzmann} does not depend on the random variable $z$. We set the initial conditions as follows:
    \begin{equation}\label{initial-condition}
        f_{0}(x,v)=\frac{\e^{-v^2}}{\sqrt{\pi}}(1+a_0\sin(2\pi x)).
    \end{equation}
    
\noindent where $a_0 \in [0,1]$ is a constant. We randomly generate 10 values for $a_0$(assuming a uniform distribution). For each initial condition, we solve for reference solutions of the moments as described in \ref{subsec:data prep}. We then use all these reference moments up to time $t=0.4$ as our training data. Once the network is properly trained, we test the performance of our model at time $t=0.5$ with a new initial condition given by $a_0=0.9$. With this setup, we can analyze the generalizability of our model across different initial conditions and various time spots. 

We compare the results with two constant choices for the collision frequency $\sigma$: $\sigma=2$ and $\sigma=10$, and three different number of moments $N$, including $N=1$, $N=3$ and $N=5$. For $\sigma = 10$, in Fig.~\ref{N1,sigma10}, \ref{N3,sigma10} and \ref{N5,sigma10}, we show the numerical solutions of $m_0$ and $m_1$ at $t=0.5$ for $N=1,3,5$, respectively. It can be seen that all closure models can achieve reasonably good approximation results with this large collision frequency. This observation can be partly explained by the following arguments. When $\sigma$ is a constant in \eqref{Qf}, we can view it as $\sigma = \frac{1}{\varepsilon}$, where $\varepsilon$ is the dimensionless Knudsen number characterizing the ratio of particle mean free path over the domain size. When $\sigma$ is large, $\varepsilon$ is small and the model equation converges to its macroscopic limit, where a closed moment system can be derived. Hence, we can expect our moment system to be properly closed with large values for $\sigma$. 

\begin{figure}[!htbp]
    \includegraphics[width=1\textwidth]{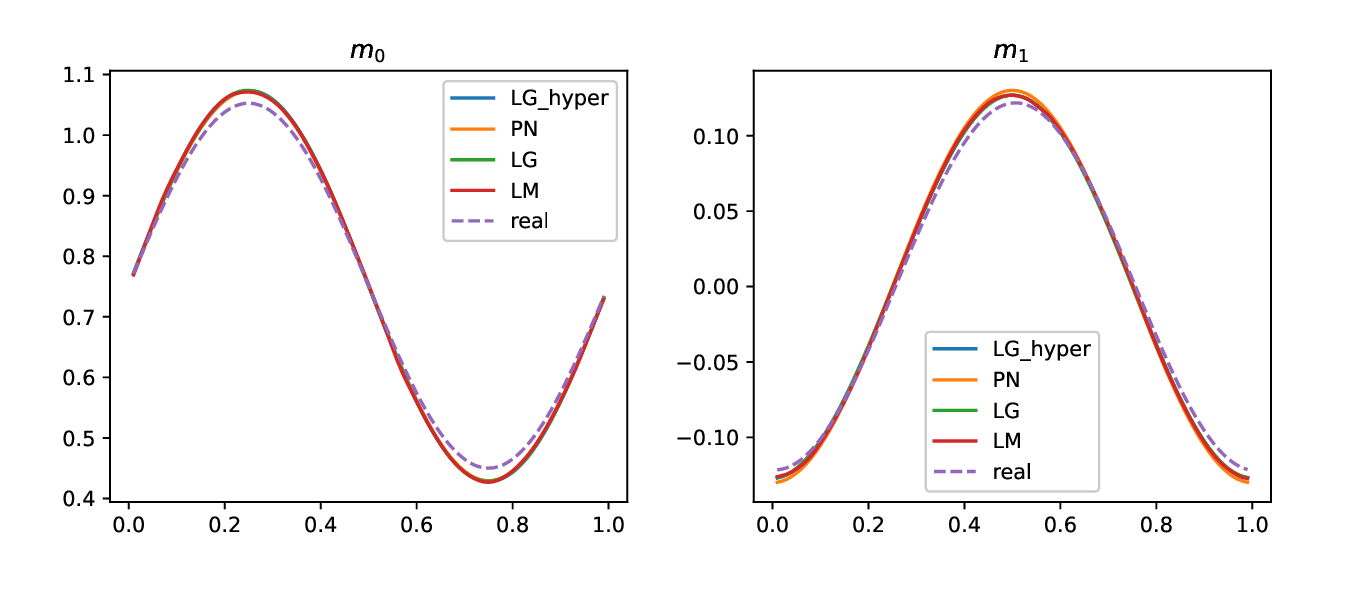}
    \caption{Comparison between the benchmark results by solving kinetic equation (``real") and predicted results using $P_N$ model (``PN"), LM model (``LM"), LG model (``LG") and LG model involving hyperbolicity (``LG$\_$hyper") with $N=1$, $\sigma=10$ at $t=0.5$.}
    \label{N1,sigma10}
\end{figure}
\begin{figure}[!htbp]
    \includegraphics[width=1\textwidth]{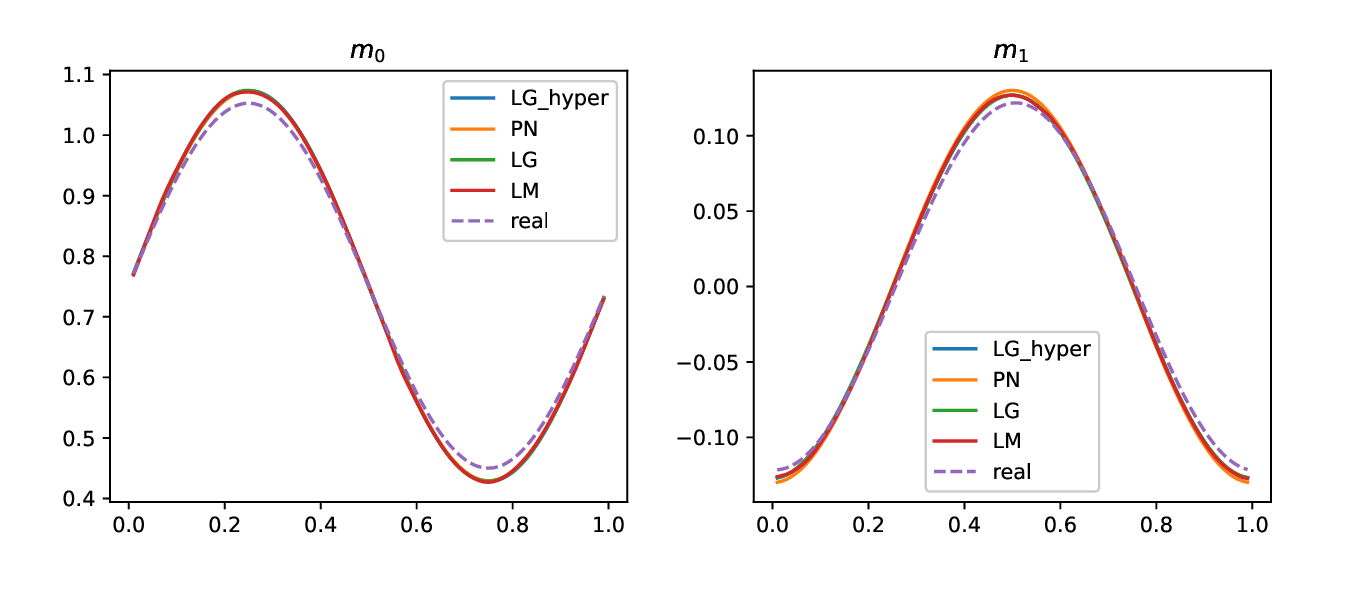}
    \caption{Comparison between the benchmark results by solving kinetic equation (``real") and predicted results using $P_N$ model (``PN"), LM model (``LM"), LG model (``LG") and LG model involving hyperbolicity (``LG$\_$hyper") with $N=3$, $\sigma=10$ at $t=0.5$.}
    \label{N3,sigma10}
\end{figure}
\begin{figure}[!htbp]
    \includegraphics[width=1\textwidth]{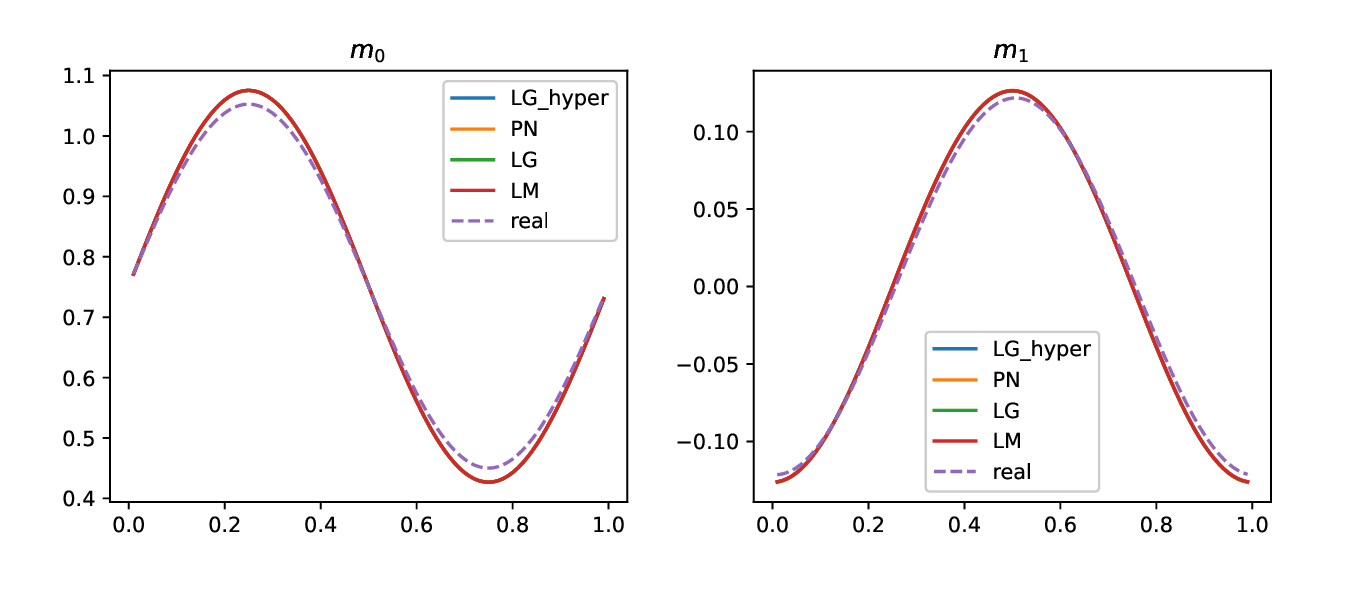}
    \caption{Comparison between the benchmark results by solving kinetic equation (``real") and predicted results using $P_N$ model (``PN"), LM model (``LM"), LG model (``LG") and LG model involving hyperbolicity (``LG$\_$hyper") with $N=5$, $\sigma=10$ at $t=0.5$.}
    \label{N5,sigma10}
\end{figure}

On the other hand, when $\sigma$ is small($\varepsilon$ is large), the model lies in a kinetic regime, where our knowledge about moment closure is quite limited. In this case, we can not expect to close the moment system properly using simple methods like $P_N$, especially when the number of moments we use is small. This assertion can be verified in Fig.~\ref{N1,sigma2}, \ref{N3,sigma2} and \ref{N5,sigma2}, where we show the numerical profile of $m_0$ and $m_1$ for smaller $\sigma =2$ with $N=1,3,5$ respectively. 
    
When $N=5$ in Fig.~\ref{N5,sigma2}, there is no clear distinction between the prediction by different methods from the reference solution. When $N=3$ in Fig.~\ref{N3,sigma2}, the errors for the $P_N$ and LM model start to blow up, while those for the LG model are visible but insignificant compared to the other models, and the LG with hyperbolicity remain to be accurate. When $N=1$ in Fig.~\ref{N1,sigma2}, even the solutions obtained by the LG model become oscillatory starting at $t=0.4$. However, when the hyperbolic condition is added during the training process, the behavior of the solutions gets regulated and the model achieves a reasonable approximation to the benchmark solution obtained from the kinetic equation. 
    \begin{figure}[!htbp]
        \includegraphics[width=1\textwidth]{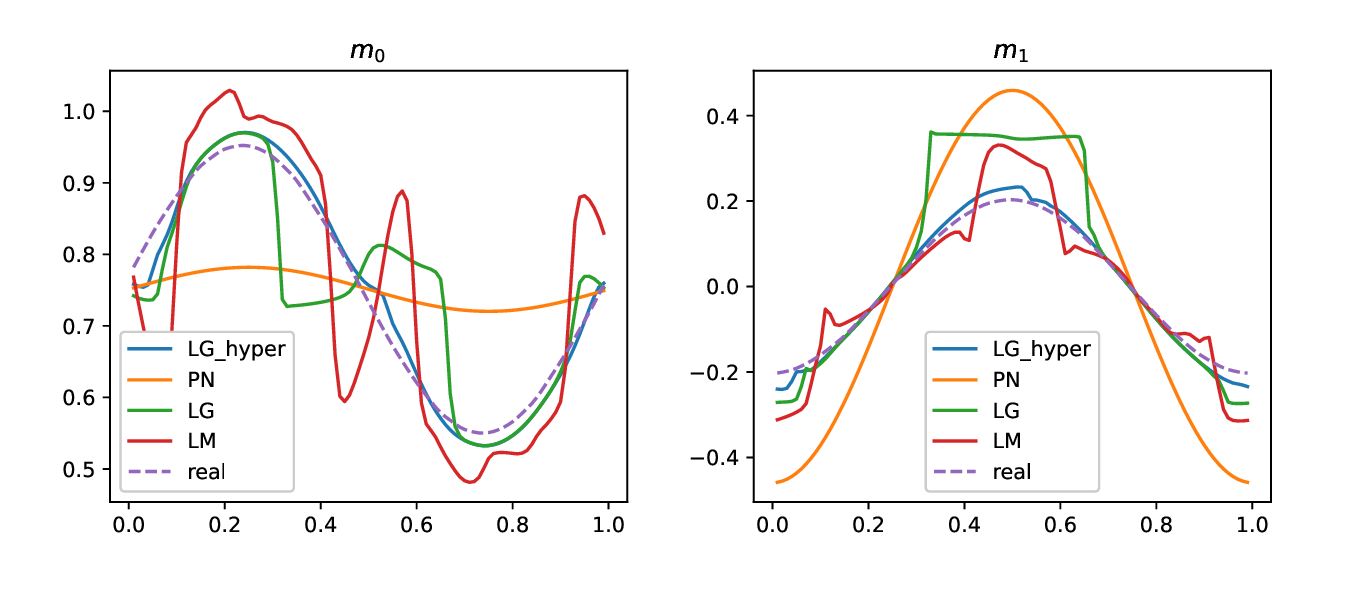}
        \caption{Comparison between the benchmark results by solving kinetic equation (``real") and predicted results using $P_N$ model (``PN"), LM model (``LM"), LG model (``LG") and LG model involving hyperbolicity (``LG$\_$hyper") with $N=1$, $\sigma=2$ at $t=0.4$.} 
        \label{N1,sigma2}
    \end{figure}
    \begin{figure}[!htbp]
        \includegraphics[width=1\textwidth]{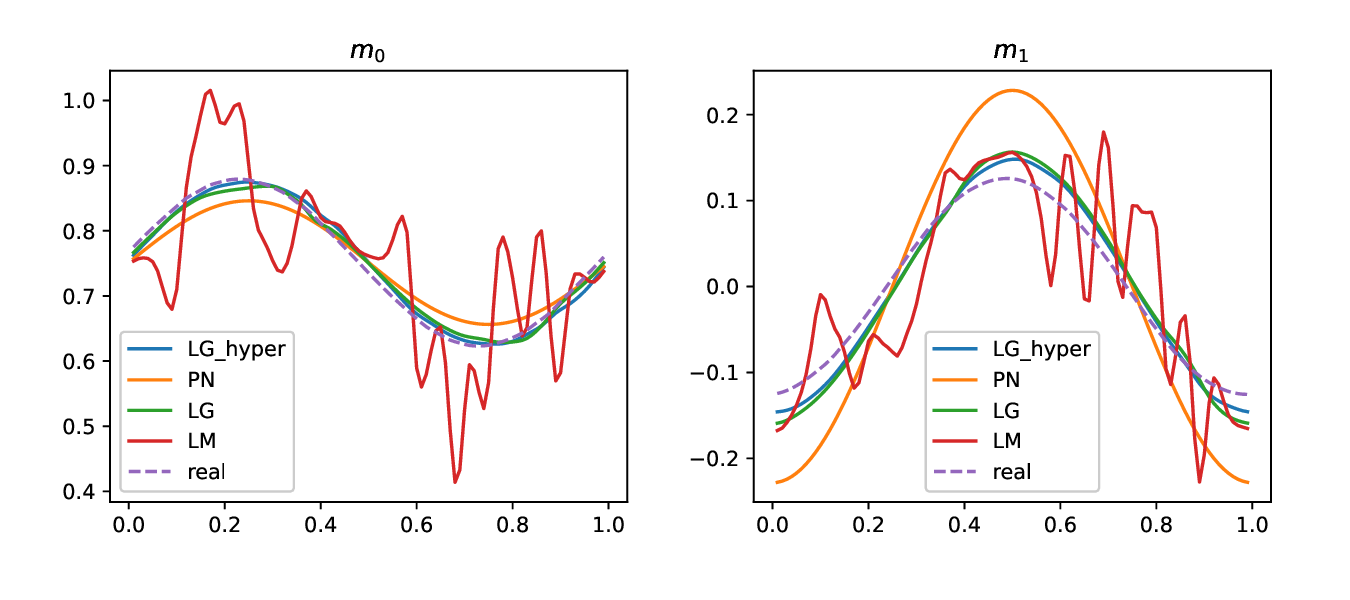}
        \caption{Comparison between the benchmark results by solving kinetic equation (``real") and predicted results using $P_N$ model (``PN"), LM model (``LM"), LG model (``LG") and LG model involving hyperbolicity (``LG$\_$hyper") with $N=3$, $\sigma=2$ at $t=0.5$.}
        \label{N3,sigma2}
    \end{figure}
    \begin{figure}[!htbp]
        \includegraphics[width=1\textwidth]{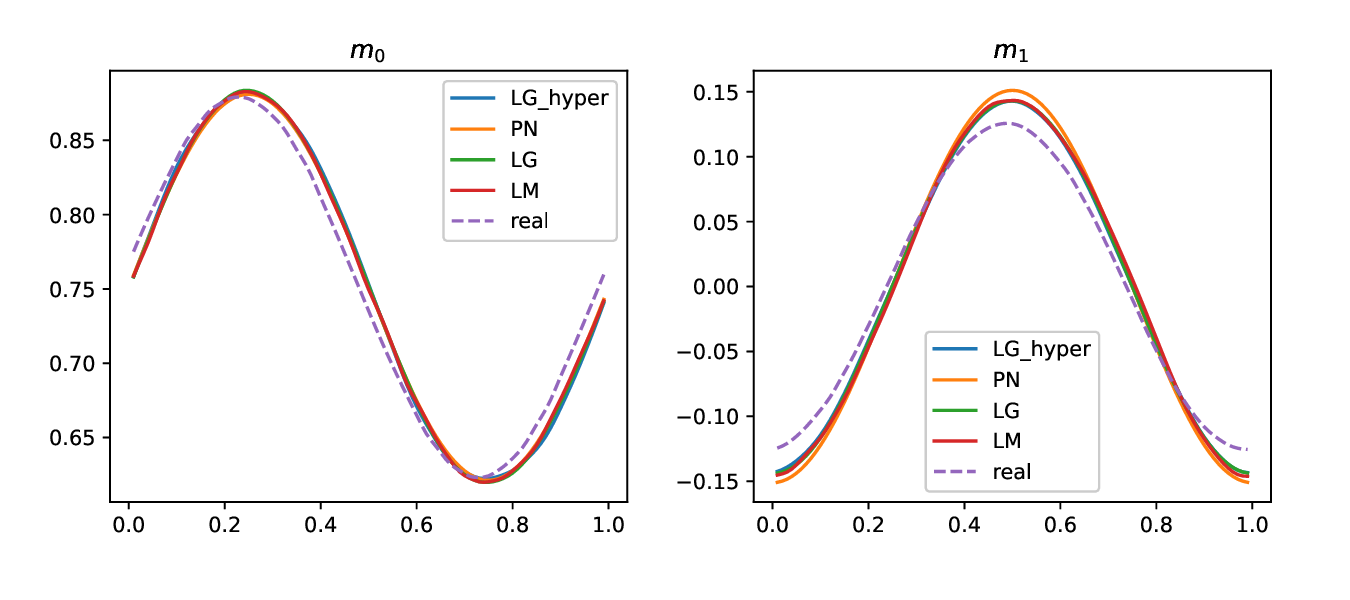}
        \caption{Comparison between the benchmark results by solving kinetic equation (``real") and predicted results using $P_N$ model (``PN"), LM model (``LM"), LG model (``LG") and LG model involving hyperbolicity (``LG$\_$hyper") with $N=5$, $\sigma=2$ at $t=0.5$.}
        \label{N5,sigma2}
    \end{figure}
    
The incompetence of the LG model in the case $\sigma=2, N=1$ can be partially explained if we analyze the relative $L^2$ error when predicting $\partial_{x} m_{N+1}$ using our neural network, as shown in Fig.~\ref{relative_l2_various_sigma_N1}. It can be observed that when $N=1, \sigma=2$, the saturated $L^2$ training error is about $0.08$, which is considerably more significant than for other choices of $N$ and $\sigma$. This explains the large deviation from the reference solution when using LG method in this case. We can control the oscillation by adding hyperbolicity to our model, but the prediction of $\partial_{x} m_{N+1}$ still remains inaccurate, counting for the relatively large errors of the method LG$\_$hyper in Fig.~\ref{N1,sigma2}. 

\begin{figure}[!htbp]
    \centering \includegraphics[width=0.6\textwidth]{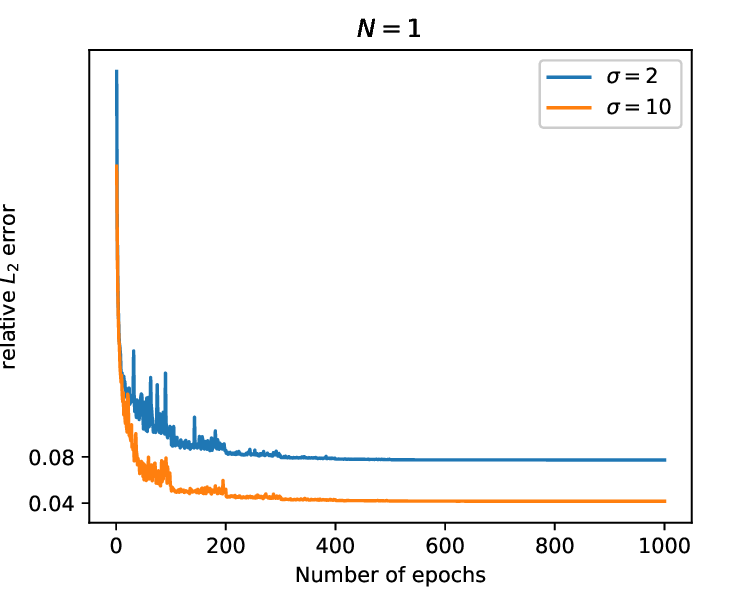}
    \caption{Relative $L^2$ errors of \eqref{relative-L2-det} for $N=1$ with $\sigma=2,10$.}
    \label{relative_l2_various_sigma_N1}
\end{figure}

\subsection{\textbf{Test II: Random collision frequency}}

We now study the case with collision frequency $\sigma$ involving randomness. We set $\sigma(z)=2+z$ where $z$ follows the exponential distribution with parameter $\lambda =1$, i.e., $\pi(z)=e^{-z}$ for $z \in [0,\infty)$. 
In this case, the gPC-basis functions $\{ \phi_{i}(z) \}_{i \geq 0}$ are given by the Laguerre polynomials with the recurrence relation: 
    \begin{equation}\label{Laguerre-recursion}
        \phi_{i+1}(z)= \frac{2i+1-z}{i+1}\phi_{i}(z) -\frac{i}{i+1}\phi_{i-1}(z), \quad i \geq 1,
    \end{equation}
    and $\phi_{0}(z)=1, \phi_{1}(z)=-z+1$.

In Fig.~\ref{random-collision-kernel}, we show the numerical simulation of the mean ($m_0^0, m_1^0$) of $m_0$ and $m_1$ with $N=3$ at $t=0.5$. In this example, the initial condition assumes the form of \ref{initial-condition}, with $a_0=0.9$. The training process is similar to that in the deterministic case, except that we replace the reference moments with their stochastic Galerkin counterparts. When performing the Galerkin expansion, we choose the order of truncation to be $K=3$.  One can observe that our proposed LG model performs much better than the $P_N$ model, especially when predicting the mean of $m_1$. Performances of our method among various choices of $K$ and $N$ will be examined in test III(b) below.

\begin{figure}[!htbp]
    \includegraphics[width=1\textwidth]{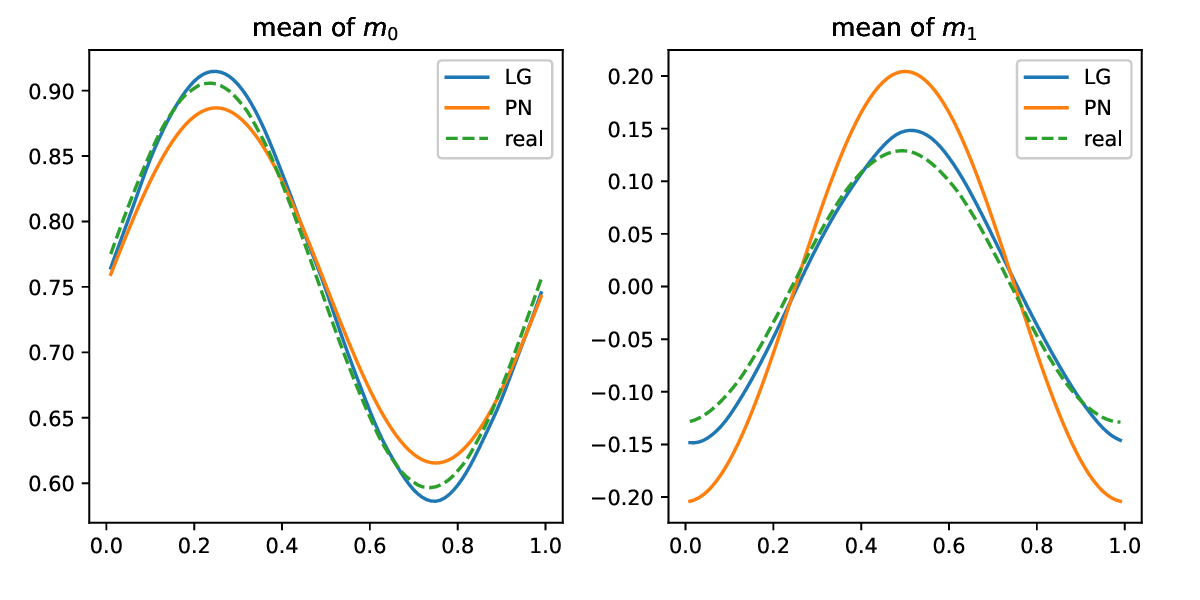}
    \caption{Comparison of the mean ($m_0^0, m_1^0$) of $m_0$ and $m_1$ between the benchmark results by solving kinetic equation (``real") and predicted results using $P_N$ model (``PN") and our proposed LG model (``LG") for random collision frequency with $N=3$ at $t=0.5$.}
    \label{random-collision-kernel}
\end{figure}

\subsection{\textbf{Test III: Random initial data}}
In the following two tests, we consider the initial data containing uncertainties. The collision frequency is constant and set as $\sigma=2$ in both tests below. 

\vspace{4mm}
\noindent\textbf{Test III (a): } 
    We study the problem with uncertain initial data, which is given by 
    \begin{equation}\label{random-initial-data-1}
        f_{0}(x,v,z)=\frac{\e^{-v^2}}{\sqrt{\pi}}(3+(1+z)\sin(2\pi x)),
    \end{equation}
    where $z$ follows the uniform distribution on $[-1,1]$, i.e., $\pi(z)=\frac{1}{2}$ for $z \in I_z = [-1,1]$. 
    In this case, the gPC-basis functions $\{ \phi_{i}(z) \}_{i \geq 0}$ are given by the Legendre polynomials in the recurrence relation: 
    \begin{equation}\label{Legendre-recursion}
        \phi_{i+1}(z)= \frac{2i+1}{i+1}z\phi_{i}(z) -\frac{i}{i+1}\phi_{i-1}(z), \quad i \geq 1.
    \end{equation}
    with $\phi_{0}(z)=1, \phi_{1}(z)=z$. 
The basis functions are normalized in our simulation.  

In Fig.~\ref{random-initial-1}, we show the numerical solutions of the mean ($m_{0}^{0}, m_{1}^{0}$) of $m_0$ and $m_1$ with $N=3$ at $t=0.5$. We again set the order of truncation to be $K=3$. Our proposed LG model accurately reproduces the mean of $m_0$, while the error in the mean of $m_1$ is noticeable but significantly smaller compared to that in the $P_N$ model.

\begin{figure}[!htbp]
    \includegraphics[width=1\textwidth]{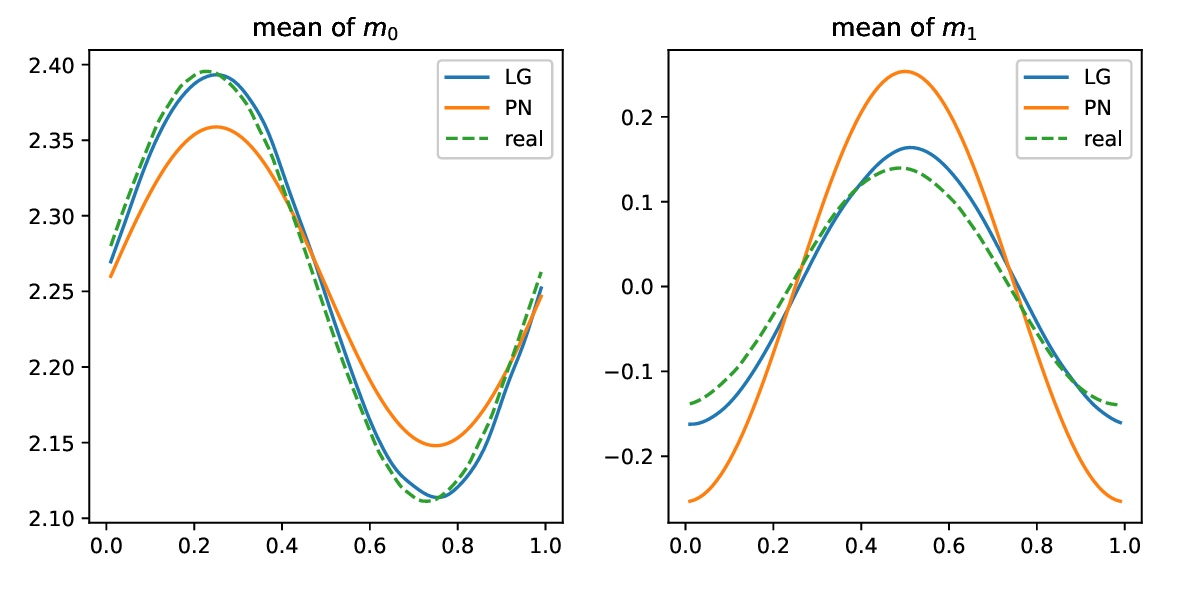}
    \caption{Comparison of the mean ($m_0^0, m_1^0$) of $m_0$ and $m_1$ between the benchmark results by solving kinetic equation (``real") and predicted results using $P_N$ model (``PN") and our proposed LG model (``LG") for random initial data \eqref{random-initial-data-1} with $N=3$ at $t=0.5$.}
    \label{random-initial-1}
\end{figure}

\noindent\textbf{Test III (b): } 
   In the last example, we assume the initial data contains uncertainty and is given as
    \begin{equation}\label{random-initial-data-2}
        f_{0}(x,v,z)=\frac{\e^{-v^2}}{\sqrt{\pi}}(2+\sin(2\pi x(1+z))),
    \end{equation}
The setting for the random variable $z$ and the choices of gPC basis functions are the same as in Test III (a) above. In Fig.~\ref{random-initial-2}, we illustrate the comparison for the mean ($m_0^0, m_1^0$) approximated by the $P_N$ model and our proposed LG model, with $N=3$ and truncation order $K=3$. The $P_N$ model exhibits significant errors in solving the system, whereas the LG model effectively captures the random effects, even as the true solutions exhibit more oscillatory behavior compared to the previous test. The deviation in the LG model, particularly for $m_1$, becomes noticeable; however, the errors remain substantially smaller than those observed in the $P_N$ model. 

\begin{figure}[!htbp]
    \includegraphics[width=1\textwidth]{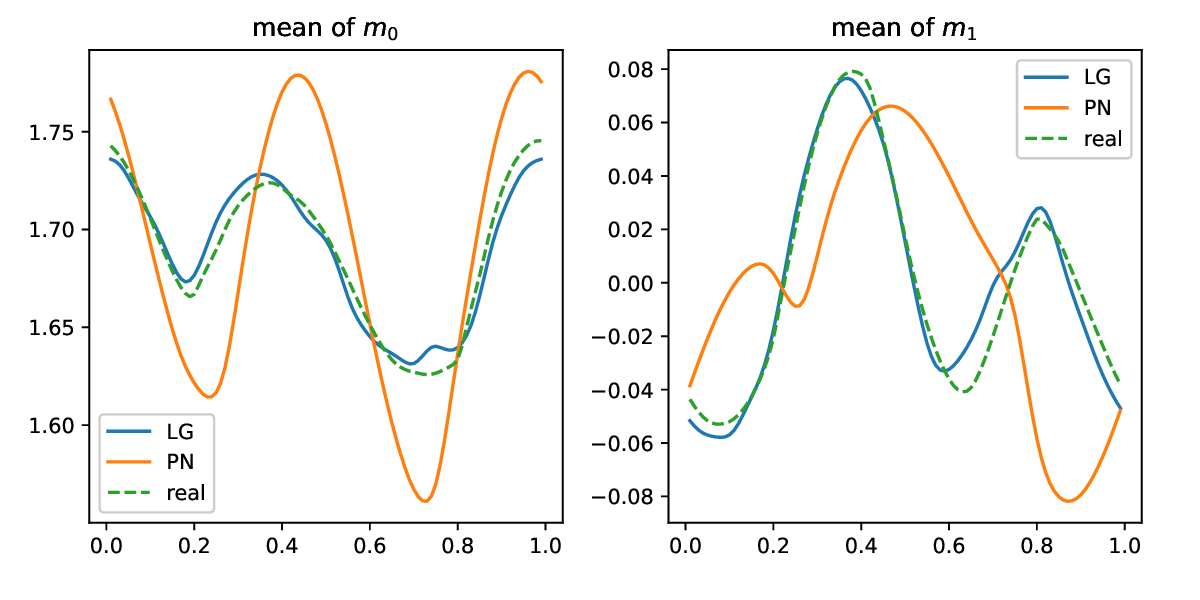}
    \caption{Comparison of the mean ($m_0^0, m_1^0$) of $m_0$ and $m_1$ between the benchmark results by solving kinetic equation (``real") and predicted results using $P_N$ model (``PN") and our proposed LG model (``LG") for random initial data \eqref{random-initial-data-2} with $N=3$, $K=3$ at $t=0.5$.}
    \label{random-initial-2}
\end{figure}

Lastly, we examine the performance of our method in the context of Test III (b), when the number of moments $N$ and the order of gPC truncation $K$ vary. In Fig.~\ref{N3_sigma2_sto_3_various_K}, we fix $N=3$ and let $K$ vary from $1$ to $5$. When $K$ is small, e.g. $K=1$, our solution exhibits large spurious oscillations. As we increase $K$, the oscillations get regulated, but the improvement becomes increasingly marginal when $K$ grows beyond $3$. Therefore, we set $K=3$ for our tests to balance between the precision of the solutions and the computational costs. Similar patterns can be observed when we fix $K=3$ and let $N$ vary from $1$ to $5$, as shown in Fig.~\ref{K3_sigma2_sto_3_various_N}. 

    \begin{figure}[!htbp]
    \includegraphics[width=1\textwidth]{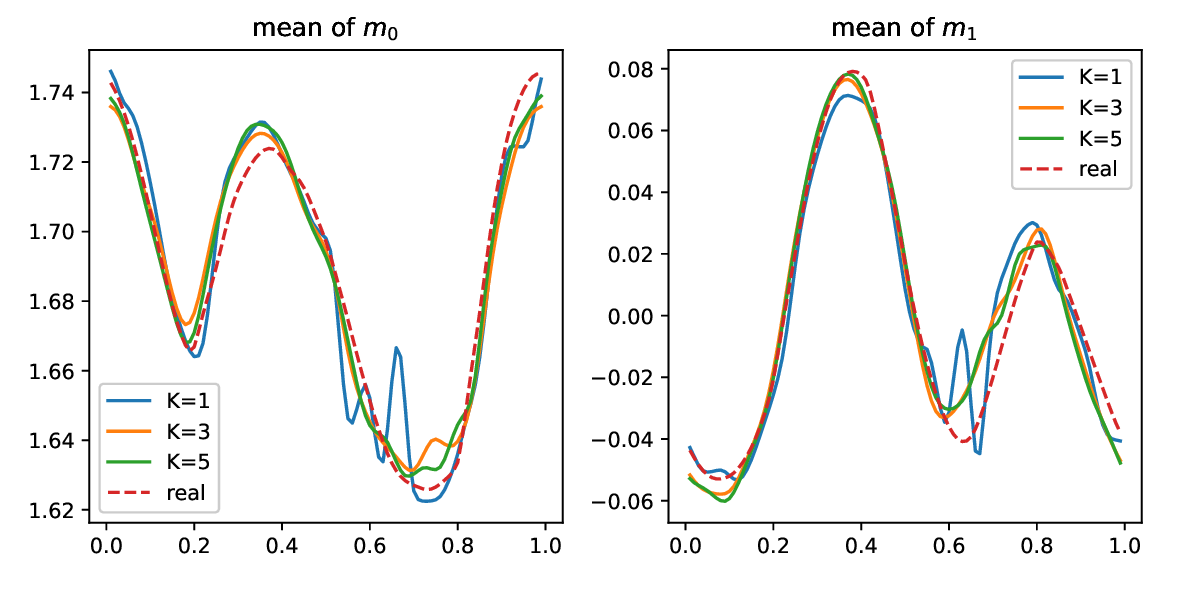}
    \caption{Comparison of the mean ($m_0^0, m_1^0$) of $m_0$ and $m_1$ when using different values of $K$ to truncate the gPC expansion. The moment system is closed using the LG model for $N=3$, $t=0.5$ with random initial data \eqref{random-initial-data-2}.}
    \label{N3_sigma2_sto_3_various_K}
    \end{figure}

    \begin{figure}[!htbp]
    \includegraphics[width=1\textwidth]{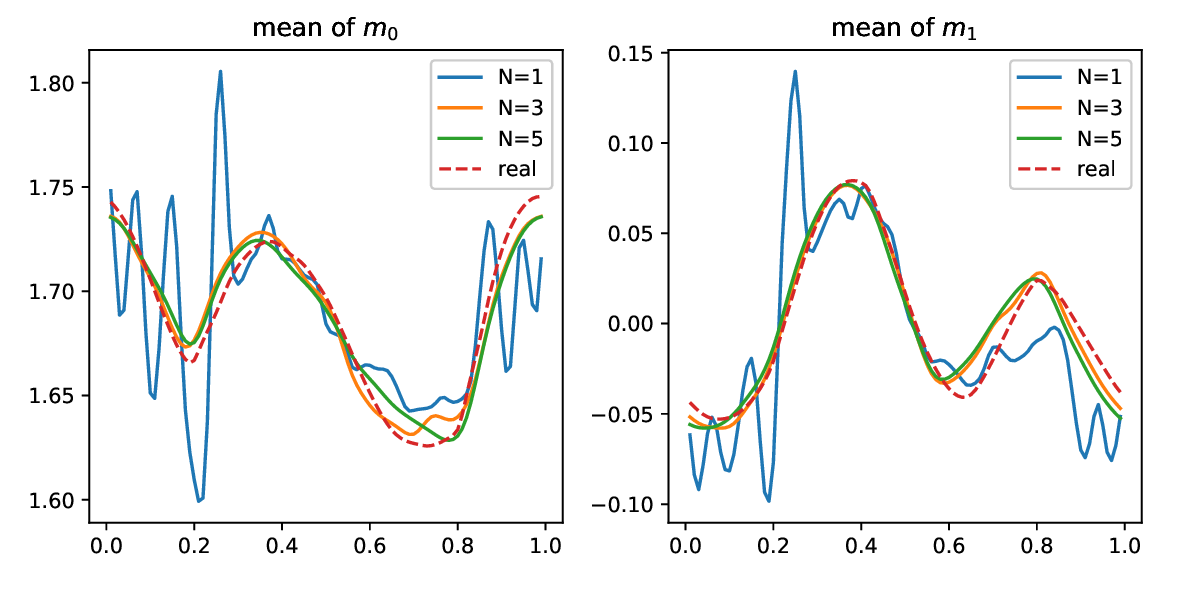}
    \caption{Comparison of the mean ($m_0^0, m_1^0$) of $m_0$ and $m_1$ when using different values of $N$ to close the moment system under the LG model at $t=0.5$ with random initial data \eqref{random-initial-data-2}. The truncation order is fixed at $K=3$ for various $N$.}
    \label{K3_sigma2_sto_3_various_N}
    \end{figure}

\section{Conclusion}
\label{sec:conclusion}
In this work, we develop a ML-based moment closure model for the linear Boltzmann equation, addressing both deterministic and stochastic settings. By using neural networks to approximate the spatial gradient of the unclosed highest-order moment, our approach achieves effective training to close the moment system.
To guarantee its global hyperbolicity and stability, we imposed constraints for ensuring the symmetrizable hyperbolicity. For the stochastic problem, we incorporated a gPC-based SG method to discretize the random variables, transforming the problem into one that is similar to the deterministic setting. 
Several numerical experiments validate the proposed framework, highlighting its stability and accuracy in achieving reliable moment closures for kinetic problems with (or without) uncertainties. These results underscore the potential of incorporating ML techniques to the moment closure of the more complicated nonlinear Boltzmann equation in our future work. \\

\noindent \textbf{Data Availability} Enquiries about data availability should be directed to the authors.

\section*{Declarations}
\textbf{Competing interests} The authors have not disclosed any competing interests.

\appendix

\section*{\centering \textnormal{APPENDIX.} \textbf{Parameterization of random inputs}}

In this appendix, we discuss how to manage the randomness existing in our model. The key step is to parameterize the random inputs by a finite set of independent random variables to make computational simulations plausible. If the random inputs are already given in the form of finitely many random parameters with proper probability distribution, e.g. jointly Gaussian, then parametrization is straightforward, e.g. using Cholesky decomposition. However, In many cases, the random inputs are formulated by random processes, which are often characterized by a continuous index $t \in T$. Then we need to apply dimension reduction techniques to approximate the processes using finitely many random variables. One of the most widely used techniques in this regard is the Karhunen-Loeve(KL) expansion, see \cite{Hajek}, \cite{Xiu}, \cite{Sullivan}. 

For a random process $ \{Y_{t}(\omega) \}_{t \in T}$(we use the notation $Y_t$ as an abbreviation) with mean $\mu_{Y}(t)$ and autocorrelation function $R_{Y}(t,s)=\mathbb{E}[Y_tY_s]$, its KL expansion, if exists, is given by 

\begin{equation}\label{def of KL expansion}
    Y_{t}(\omega) = \sum_{i=1}^{\infty} \psi_{i}(t) Y_{i}(\omega)
\end{equation}

\noindent where the series converges in the mean square sense(m.s.), $\{\psi_{i}\}$ is an orthonormal family of functions in $L^2(T)$ and $Y_{i}$'s are mutually orthogonal, i.e., $\mathbb{E}[Y_{i}Y_{j}]=0 ~~\forall i \neq j$. The existence of KL expansion for $Y_{t}$ is guaranteed by Mercer's theorem(see \cite{Hajek} for more details), provided that the random process $Y_{t}$ is m.s. continuous, i.e., $R_Y(t,s)$ is continuous over $T \times T$. 

We now analyze how to derive $\psi_{i}$ and $Y_i$ if the KL expansion for $Y_t$ exists. The KL expansion for $Y_t$ can be viewed as an analogue of decomposing $f \in L^2(T)$ with respect to an orthonormal family $\{\psi_{i}\}$, i.e., 

\begin{equation*}
    f=\sum_{i=1}^{\infty} f_{i}\psi_{i}
\end{equation*}  

\noindent where $f_i = \int_{T} f\psi_{i} ~dt \in \mathbb{R}$ are the Fourier coefficients. In the KL setting, once an orthonormal family $\{\psi_{i}\}$ is chosen, we can define the "Fourier coefficients" $Y_i$ analogously by 

\begin{equation}\label{Fourier coefficient for Y_i}
    Y_i(\omega) = \int_{T} Y_{t}(\omega)\psi_{i}(t) ~dt.
\end{equation}

\noindent The only difference is that the coefficients for the KL expansion are random variables instead of real numbers. The challenge now is to pick $\{\psi_{i}\}$ properly so that the coefficients $\{Y_{i}\}$ are mutually orthogonal. This can be accomplished by the following lemma, whose proof is included in \cite{Hajek}.  

\begin{lemma}\label{eigenvalue problem for psi_i}
    Suppose $Y_t$ is m.s. continuous and \eqref{def of KL expansion} holds for $Y_t$ with $\{\psi_i\}$ orthonormal and $\{Y_i\}$ not necessarily mutually orthogonal. Then it is a KL expansion(i.e., $Y_i$'s are mutually orthogonal) if and only if $\psi_i$'s are eigenfunctions of $R_Y$: 

    \begin{equation*}
        R_Y(\psi_i) = \lambda_i \psi_i, 
    \end{equation*}

    \noindent where $R_Y$ is an operator on $L^2(T)$ given by $R_Y(\psi)(t)=\int_{T} R_Y(t,s)\psi(s) ~ds$ for any $\psi(t) \in L^2(T)$. In case \eqref{def of KL expansion} is a KL expansion, the eigenvalues are given by $\lambda_i = \mathbb{E}[|Y_i|^2]$. 
\end{lemma}

\noindent In summary, if $Y_t$ is m.s. continuous, its KL expansion can be established by first solving an eigenvalue problem related to the autocorrelation function $R_Y(t,s)$ to obtain an orthonormal family $\{ \psi_i \}$, followed by a computation of the "Fourier coefficients" $Y_i(\omega)$ associated to this orthonormal family. Once the KL expansion is established, the analysis of $Y_t$ can be naturally transformed into the analysis of the coefficients $Y_i$. 

For practical purposes, we need to truncate the series appeared in \eqref{def of KL expansion} to obtain a finite dimensional parametrization of the random process, i.e., 

\begin{equation}\label{KL expansion with truncation}
    Y_{t}(\omega) \approx \sum_{i=1}^{d} \psi_{i}(t) Y_{i}(\omega), \hspace{0.8cm} d \geq 1
\end{equation}

\noindent In most situations, the eigenvalues $\lambda_{i}$ as appeared in lemma \ref{eigenvalue problem for psi_i} will decay as $i$ increases. Hence we can choose the truncation order $d$ based on the decay rate of the eigenvalues. For more details, we refer the readers to \cite{Xiu}. Once \eqref{KL expansion with truncation} is established, we can represent the random process $Y_t$ by finitely many orthogonal random variables $Y_i$ as we desire. Note that in general $Y_i$'s are not mutually independent, unless additional assumptions on $Y_t$ are made, e.g. $Y_t$ is a Gaussian process with zero mean. We will not pursue further in this direction and shall be content with finite representation of $Y_t$ by orthogonal random variables. Some remarks are in order. 

\begin{remark}
    The condition of m.s. continuity to guarantee the existence of a KL expansion is not very restrictive. Many random processes we use for modeling, e.g. Brownian motion and Poisson process, satisfy this property. 
\end{remark}

\begin{remark}
    We assume the distribution of the random process $Y_t$ is prescribed. Hence we can derive the probability distributions of $Y_{i}$ using \eqref{Fourier coefficient for Y_i}. This is especially straightforward and useful when $Y_{t}$ is a Gaussian process. 
\end{remark}

\begin{remark}
    The truncated KL expansion \eqref{KL expansion with truncation} identifies the "most accurate" $d$-dimensional approximation of $Y_t$ in the sense that it minimizes $\mathbb{E}[\| Y_t - Z_t\|^2]$ over all $d$-dimensional random processes $Z_t$. A random process $Z_t$ is said to be $d$-dimensional if it has the form $Z_{t}(\omega) = \sum_{i=1}^{d} \phi_{i}(t) Z_{i}(\omega)$ for any $d$ random variables $Z_1,...,Z_d$ and functions $\phi_1,...,\phi_d$. 
\end{remark}

\bibliographystyle{spmpsci}
\bibliography{Ref.bib}

\end{document}